\documentclass[a4paper, 11pt]{amsart}
\usepackage[english]{babel}
\usepackage[utf8]{inputenc}
\usepackage{fancyhdr}
\usepackage{xcolor}
\usepackage{comment}

\usepackage{amssymb}
\usepackage{amsmath} 
\usepackage{amsfonts}
\usepackage{amsthm} 
\usepackage{mathrsfs}

\pagecolor{white}

\usepackage{enumitem}
\usepackage{bbm}
\usepackage{todonotes}
\newtheorem{thm}{Theorem}[section]
\newtheorem{remark}[thm]{Remark}
\newtheorem{conj}[thm]{Conjecture}
\newtheorem{prop}[thm]{Proposition}

\newtheorem{lemma}[thm]{Lemma}
\newtheorem{corol}[thm]{Corollary}

\newtheorem*{prop*}{Proposition}
\newcommand{\R}{\mathbb R}
\newcommand{\N}{\mathbb N}
\newcommand{\Z}{\mathbb Z}
\newcommand{\C}{\mathbb C}
\newcommand{\eps}{\varepsilon}

\newcommand{\balpha}{\boldsymbol{\alpha}}
\newcommand{\bbeta}{\boldsymbol{\beta}}

\numberwithin{equation}{section}

\title{A weighted one-level density of the non-trivial zeros of the Riemann zeta-function}
\author[Bettin]{Sandro Bettin}
\author[Fazzari]{Alessandro Fazzari}
\address{Universit\`a  di Genova,  Dipartimento di Matematica. Via Dodecaneso 35, 16146 Genova, Italy}
\email{bettin@dima.unige.it}
\address{American Institute of Mathematics, 600 East Brokaw Road, San Jose, CA 95112, US}
\email{fazzari@aimath.org}


\DeclareMathOperator{\supp}{ supp}

\DeclareMathOperator{\meas}{ meas}

\newmuskip\pFqmuskip
\newcommand*\pFq[6][8]{%
  \begingroup % only local assignments
  \pFqmuskip=#1mu\relax
  \mathchardef\normalcomma=\mathcode`,
  \mathcode`\,=\string"8000
  \begingroup\lccode`\~=`\,
  \lowercase{\endgroup\let~}\pFqcomma
  {}_{#2}F_{#3}{\left[\genfrac..{0pt}{}{#4}{#5};#6\right]}%
  \endgroup
}
\newcommand{\pFqcomma}{{\normalcomma}\mskip\pFqmuskip}

\begin{document}

\begin{abstract}
We compute the one-level density of the non-trivial zeros of the Riemann zeta-function weighted by $|\zeta(\frac12+it)|^{2k}$ for $k=1$ and, for test functions with Fourier support in $(-\frac12,\frac12)$, for $k=2$. 
As a consequence, for $k=1,2$, we deduce under the Riemann hypothesis that $T(\log T)^{1-k^2+o(1)}$ non-trivial zeros of $\zeta$, of imaginary parts up to $T$, are such that $\zeta$ attains a value of size $(\log T)^{k+o(1)}$ at a point which is within $O(1/\log T)$ from the zero. 

%As a consequence, for $k=1,2$, we deduce under the Riemann hypothesis that $T(\log T)^{1-k^2+o(1)}$ non-trivial zeros $\frac12+i\gamma$ of $\zeta$, with $|\gamma|\leq T$, are such that $\zeta$ attains a value of size $(\log T)^{k+o(1)}$ at a point $\frac12+it$ with $|t-\gamma|\leq 1/\log T$.

\end{abstract}
%%%%%%%%%%%%%%%%%%%%%%%%%%%%%%%%%%%%%%%%%%%%%%%%%%%%%%%%%%%%%%%%%%%%%%%%%%%%
\maketitle

\section{Introduction}%\label{S1}

Linear statistics of zeros of $L$-functions are a topic of central importance in number theory. It is expected that their behaviour can be modelled by analogous statistics in random matrix theory~\cite{Montgomery,KatzSarnak,KatzSarnak2,RudnickSarnak}. However this prediction, which would have far reaching consequences (see e.g.~\cite{ConreyIwaniec}), has been proven only in limited cases. 

In this paper, we consider the average behaviour of 
\begin{equation*}
N_f(t):=\sum_\gamma f\left(\frac{\gamma-t}{2\pi/\log T}\right), \qquad |t|\asymp T>1,
\end{equation*}
where the sum is over the non-trivial zeros $\rho=\frac12+i\gamma$ of the Riemann zeta-function\footnote{Notice we are not assuming the Riemann hypothesis, so $\gamma$ is not necessarily real.} and
 $f$ is a real-valued and even test function. Notice that $\frac{2\pi}{\log T}$ is the mean spacing of the imaginary part of the zeros of the Riemann zeta-function at height $|t|\asymp T$. 
 
The mean of $N_f(t)$ is called the one-level density of the non-trivial zeros of the Riemann zeta-function. In~\cite{MSConrey} it is shown that
\begin{equation}\label{c1ld}
\frac{1}{T}\int_T^{2T}N_f(t)\,dt=\int_{-\infty}^{+\infty}W_U(x)f(x)\,dx+O\left(\frac{1}{\log T}\right),
\end{equation}
where $W_U(x):=1$ for all $x\in\R$, provided that the the support of $\hat f$ is contained in $(-2,2)$. Moreover, the same result is known without this restriction either with a smooth average over $t$~\cite{HR} or under the Riemann hypothesis~\cite{MSConrey}.  
We remark that~\eqref{c1ld} is consistent with $\{\zeta(s+i\tau)\mid \tau\in[T,2T]\}$ being a unitary (continuous) family.

In~\cite{4.} the second named author considered the analogue of the classical one-level density for families of $L$-functions where each $L$-function is weighted according to the size of (a power of) its central value. This allows one to measure the effect that large central values have in the distribution of nearby zeros. 
In the ``continuous average case'' of the Riemann zeta-function, this weighted statistic corresponds to tilting the measure $dt$ in~\eqref{c1ld} by introducing a factor of $|\zeta(\frac12+it)|^{2k}$ for $k\in\Z_{\geq0}$. In~\cite{4.} it is conjectured that as $T\to\infty$ one has
\begin{equation}\label{w1ld}
\langle N_f\rangle_{k,T}:=\frac{1}{c_kT (\log T)^{k^2}}\int_T^{2T}N_f(t)|\zeta(\tfrac12+it)|^{2k}\, dt=\int_{-\infty}^{+\infty}W_U^{k}(x)f(x)\,dx+o(1),
\end{equation}
for any even $f$ with $\hat f\in\mathcal{C}^{\infty}_c(\R)$ and certain kernels $W^{k}_U(x)$, where $c_k$ is the conjectural constant for the $2k$-th moment of $\zeta$ (see~\cite{KeaS}). Moreover, this conjecture was proven under the ratio conjecture in the case $k\leq 2$, with kernels
\begin{equation*}
\begin{split}%\label{kernels}
W_U^0(x)&=W_U(x)=1,\qquad W_U^1(x)=1-\frac{\sin^2(\pi x)}{(\pi x)^2},\\
W_{U}^2(x)&=1-\frac{2+\cos(2\pi x)}{(\pi x)^2}+\frac{3\sin(2\pi x)}{(\pi x)^3}+\frac{3(\cos(2\pi x)-1)}{2(\pi x)^4}
\end{split}\end{equation*}
which are shown to coincide with the kernels appearing in the analogous statistics for the eigenvalues of unitary random matrices. 

We refer to~\cite{4.} for a more detailed discussion on these weighted averages and these kernels, as well as for analogous conjectures for other families. We mention also~\cite{Su}, where a similar phenomenon was observed when considering the $1$-level density for symmetric power $L$-functions weighted by the central value of the corresponding symmetric square $L$-function. Very recently, the analogous of~\eqref{w1ld} in the case of Dirichlet $L$-functions was proven for $k=1$  by Sugiyama and Suriajaya \cite{Su-Su}, under the additional hypothesis that $\supp(\hat f)\subseteq (-1/3,1/3)$. We also mention the works~\cite{1.,2.} by the second named author and~\cite{BELP} by Bui, Evans, Lester and Pratt which study weighted central limit theorems for central values of families of $L$-functions.

Writing $f_\alpha(\cdot):=f(\,\cdot -\alpha)$ for any $\alpha\in\R$ and any continuous fast decaying $f$, one can easily see that for $\alpha\ll1$ under the Riemann hypothesis (RH) we have
%\footnote{RH is needed to insure that the change of variable $t\mapsto t+\gamma$ maps $\R$ into itself.}
\begin{align}
\int_T^{2T}N_{f_{\alpha}}(t)|\zeta(\tfrac12+it)|^{2k}\, dt%&=\sum_{T\leq \gamma\leq 2T}\int_{T-\gamma}^{2T-\gamma} f\left(\frac{-t}{2\pi/\log T}\right)|\zeta(\tfrac12+it+i\gamma)|^{2k}\,dt+O(T^\eps)\\
%&=\sum_{T\leq \gamma\leq 2T}\int_{\R} f\left(\frac{-t}{2\pi/\log T}\right)|\zeta(\tfrac12+it+i\gamma)|^{2k}\,dt+O(T^\eps)\\
&=\frac{2\pi}{\log T}\sum_{T\leq \Im(\rho)\leq 2T}\int_{\R}f(t)|\zeta(\rho -\tfrac{2\pi i}{\log T}(\alpha+t))|^{2k}\,dt+O(T^\eps)\notag\\
%&=\frac{2\pi}{\log T}\sum_{T\leq \gamma\leq 2T}\int_{\R}f(t)|\zeta(\rho+i\tfrac{2\pi t}{\log T})|^{2k}\,dt+O(T^\eps)
&=\frac{2\pi}{\log T}\int_{\R}f(t){M_k(-\alpha-t;T)}\,dt+O(T^\eps)\label{eqv}
\end{align}
where 
$$M_k(\alpha;T):=\sum_{T\leq\Im(\rho)\leq 2T}|\zeta(\rho+\tfrac{2\pi i \alpha}{\log T})|^{2k}.$$ 
Thus $\langle N_f\rangle_{k,T}$ can be seen as a version of $M_k(-\alpha;T)$ where the contribution of each (shifted) zero has been smoothed by a short average. In fact, under the Riemann hypothesis, the asymptotic formulas for $\langle N_f\rangle_{k,T}$, for all $f$ with $\hat f\in\mathcal{C}^{\infty}_c(\R)$, and for $M_k(\alpha;T)$, for all $\alpha\in \R$ are equivalent. Indeed, it's clear that the latter implies the former by integration, whereas for the opposite direction it suffices to take $f$ that approximates a Dirac delta function. We remark that this connection between moments over zeros and weighted $1$-level densities is only available for continuous families.

Gonek \cite{GonekInventiones} proved that $M_1(\alpha;T)\sim W_{U}^1(\alpha)T\log T$ under RH. By~\eqref{eqv} one then immediately obtains a conditional proof of~\eqref{w1ld} in the case $k=1$.

No asymptotic formula for $M_k$ is known for $k>1$, but Hughes conjectured an asymptotic formula for $M_k(\alpha;T)$ for all $k\in\N$ which is (conditionally) equivalent to~\eqref{w1ld}, as shown in Appendix~\ref{Appendix}.

As it is often the case, it is convenient to work with a smoothed version of $\langle N_f\rangle_{k,T}$, for which one naturally expects the following smoothed version of~\eqref{w1ld} to hold.
\begin{conj}\label{smtco}
Let $k\in\N$, $f$ even with $\hat f\in\mathcal{C}^{\infty}_c(\R)$. Let $\phi$ be a smooth function of compact support in $\R_{>0}$ and let $\tilde\phi(s):=\int_{\R}\phi(x)x^{s-1}\,dx$ be its Mellin transform. Then, as $T\to\infty$ we have
\begin{align}\label{c_sm}
\frac{1}{c_kT (\log T)^{k^2}} \int_{\R}N_f(t)|\zeta(\tfrac12+it)|^{2k}\phi\Big(\frac{t}{T}\Big)\, dt &=\tilde\phi(1)\int_{\R}{f}(x) W_U^k(x)\, dx+O((\log T)^{-1}).
\end{align}
\end{conj}

In this note we shall prove Conjecture~\ref{smtco} for $k=1$ in a more precise version. We are also able to prove the conjecture in the case $k=2$, provided the support of $\hat f$ is sufficiently small. In both cases one could remove the additional smoothing under the assumption of the Riemann hypothesis.
%\red{\ \\{} [PROOF: Let $\eta$ be even, real valued, non-negative, with $\eta(x)\geq1$ for $|x|\leq1$, and with $\hat \eta$ smooth, of small compact support. For $\|\cdot\|_{\infty,n}$  the infinity norm on $[n-1,n+1]$, let 
%$g(x):=\sum_{n\in\Z}\eta(x-n)\|f\|_{\infty,n}$. Then $g$ has the same Fourier support of $\eta$ and $|f(x)|\leq g(x)$. Now, let $\phi^{\pm}$ a smooth approximation from above/below of $\chi_{[1,2]}$. Then 
%\begin{align*}
%\bigg|\int_{\R}N_f(t)|\zeta(\tfrac12+it)|^{2k}(\phi^+({t}/{T})-\chi_{[1,2]}(t/T))\,dt\bigg|&\leq \int_{\R}N_g(t)|\zeta(\tfrac12+it)|^{2k}(\phi^+({t}/{T})-\chi_{[1,2]}(t/T))\,dt\\
%&\leq \int_{\R}N_g(t)|\zeta(\tfrac12+it)|^{2k}(\phi^+({t}/{T})-\phi^-({t}/{T}))\,dt\\
%&=c_kT(\log T)^{k^2}(\tilde\phi^+(1)-\tilde\phi^-(1))\int_{\R}{g}(x) W_U^k(x)\, dx\\
%&+O_{f,\phi,k}(T(\log T)^{k^2-1})
%\end{align*}
%Letting $\Phi^{\pm}\to\chi_{[1,2]}$ sufficiently slowly we obtain that the above is $o(T(\log T)^{k^2})$.]
%}

\begin{thm}\label{mt2}
Conjecture~\ref{smtco} holds for $k=1$. It also holds for $k=2$ provided that the support of $\hat f$ is contained in $(-\frac12,\frac12)$.
In the case $k=1$, we also have the following more precise asymptotic formula
\begin{align}\label{9febbr.1}
\int_{\R}N_f(t)|\zeta(\tfrac12+it)|^{2}\phi\Big(\frac{t}{T}\Big)\, dt &=\frac{T}{\log T} \int_{\R}{f}(x)
\bigg(\psi(T)+\mathcal G\Big(\frac{2\pi i x}{\log T}, T\Big) \bigg)dx+O_{\eps}(T^{\frac12+\eps})%\\
%&=\tilde\phi(1) T\log T  \int_{\R}{f}(x) W_U^1(x)\, dx+O_{f,\phi}(T),\notag
\end{align} 
as $T\to\infty$ and for any fixed $\eps>0$, where
\begin{equation*}
\begin{split}
\psi(T) &= \frac1T \int_{\mathbb R}\log(\tfrac{t}{2\pi})\big(\log\tfrac{t}{2\pi}+2\gamma\big)\phi\Big(\frac{t}{T}\Big)\,dt\\
%=\tilde\phi''(1) +\tilde \phi'(1) (2\gamma + 2 \log(\tfrac{T}{2 \pi})) +\tilde \phi(1) (2\gamma \log(\tfrac{T}{2 \pi}) + (\log(\tfrac{T}{2 \pi}))^2)\\
&=\tilde\phi(1) \log^2(\tfrac{T}{2 \pi})+\big(2\gamma \tilde\phi(1)+2\tilde\phi'(1)\big)\log(\tfrac{T}{2 \pi})+\tilde\phi''(1)+2\gamma \phi'(1),
\\
\mathcal G(y,T)&=2 \tilde\phi(1)\bigg(\frac{\zeta'}{\zeta}(1+y)\log\frac{T}{2\pi}+\left(\frac{\zeta'}{\zeta}\right)'(1+y)+2\gamma\frac{\zeta'}{\zeta}(1+y)\bigg) \\ 
&\quad+2\tilde\phi'(1)\frac{\zeta'}{\zeta}(1+y)-2\tilde\phi(1-y)\left(\frac{T}{2\pi}\right)^{-y}\zeta(1-y)^2. 
\end{split}
\end{equation*}
\end{thm}

It would be possible to isolate lower order terms also in the case $k=2$, but we have chosen not to do so for simplicity as the resulting expression would be rather long. In any case, we remark that by the discussion above, the case $k=2$ of the theorem can be seen as a smoothed version of an asymptotic formula for $M_2(\alpha;T)$.

Notice that, under the assumption of the Riemann hypothesis, only the $t$s such that $|\zeta(\tfrac12+it)|\asymp (\log T)^{k+o(1)}$, which form a thin subset of size $ T (\log T)^{-k^2+o(1)}$, contribute significantly to the left hand sides of~\eqref{w1ld} and~\eqref{c_sm}
(see~\cite{Sou} or Section~5 below). 
Thus, results on the weighted one-level density discussed above can be used to deduce results on large values of $\zeta$ near its zeros.

In particular, under RH, a weaker form of Conjecture~\ref{smtco} implies that there are $T(\log T)^{1-k^2+o(1)}$ non-trivial zeros $\frac12+i\gamma$ of the Riemann zeta-function for which $\zeta$ has a ``large value''  of size $(\log T)^{k+o(1)}$ nearby. 
More specifically, we have the following.

\begin{thm}\label{applt}
Assume the Riemann hypothesis and assume that Conjecture~\ref{smtco} holds for some $k\in\N$ and all even functions $f$ with Fourier support in $(-\delta,\delta)$ for some $\delta>0$. Let $Z(T):=\{\gamma\in[T,2T]\mid \zeta(\frac12+i\gamma)=0\}$ and for $U>0$ let
$$Z_k(T;U):=\Big\{\gamma \in Z(T)\mid  \Big|\max_{|u|\leq\frac1{\log T}}\log|\zeta(\tfrac12+i\gamma+iu)|-k\log\log T\Big|<U\Big\}.
$$
Then for any $U\geq \frac{5k^3\log\log T}{\sqrt{\log\log\log T}}$ we have
\begin{align}\label{bzk}
 e^{- 2k U}\ll \frac{\# Z_k(T;U)}{T(\log T)^{1-k^2}}\ll   e^{2kU }.
 \end{align}
\end{thm}

%With the same method, it would be possible to relax the inequality in the definition of $Z_k(T)$ at the cost of weakening the upper and lower bounds in~\eqref{bzk}. 
By Theorem~\ref{mt2} we immediately deduce the following corollary.

\begin{corol}\label{appltc}
Assume the Riemann hypothesis. Then~\eqref{bzk} holds for $k=1,2$.
\end{corol}

In fact, in the case of $k=1$, one can use the recent work~\cite{AB} of Arguin and Bailey instead of~\cite{Sou}. In particular, under RH we have that~\eqref{bzk} holds for any $U\geq \sqrt{\log\log T\log\log\log T}$ (see Remark~\ref{dad}).

\medskip

We remark that by~\cite{Najnudel, ABBRS} we have $\max_{|u-t|\leq1 }\log|\zeta(\tfrac12+iu)|=(1+o(1))\log \log T$ for almost all $t\in[T,2T]$. Also, by the Riemann-Von Mangoldt formula we have $Z(T)\sim \frac T{2\pi }\log T$.
% under the Lindel\"of hypothesis there are $\frac{\log T}{\pi}(1+o(1))$ zeros of $\zeta$ in each interval $[t-1,t+1]$, $t\in[T,2T]$.  
 Since $\zeta$ typically changes values at a scale of about $(\log T)^{-1+o(1)}$, then the case $k=1$ of~\eqref{bzk} says that an (approximate) maximum of $\log|\zeta(\frac12+iu)|$ in $[t-1,t+1]$ is taken next to a zero in $Z(T)$ at a rate which is of roughly the same order of magnitude as when the sets of zeros $Z(T)$ is replaced by a set of $\frac T{2\pi }\log T$ numbers taken uniformly at random in $[T,2T]$. Notice that this is in contrast with the naive expectation of the zeros of zeta having a damping effect on nearby values. Similar considerations could be made for larger values of $k$.

Theorem~\ref{applt} leaves open the problem of determining an asymptotic formula for $\# Z_k(T; U)$ and $U=\sqrt{\log\log T\log\log\log T}$, say. In fact, for any $v>0$ one could more generally consider the set $Z_k(T,v;U)$ where the maximum is taken over $|u|\leq\frac v{\log T}$. The asymptotic expansion of $W_U^{k}(x)$ at $x=0$ given in~\cite{4.}  suggests that $\# Z_k(T,v;U)$ decrease proportionally to $v^{2k+1}$ as $v$ goes to zero sufficiently slowly as $T\to\infty$. It would be nice to be able to understand whether this is indeed the case.

\medskip

The proof of Theorems~\ref{mt2} is based on the following propositions on the second and fourth twisted moments of $\zeta$. Since they could be of independent interest, we state them here.

\begin{prop}\label{TwSM}
Let $\phi:\R\to\R$ be of compact support in $\R_{>0}$.
Let $A(s)=\sum_{n\leq N}a_nn^{-s}$ be a Dirichlet polynomial of length $N\leq T^\vartheta$ for any fixed $\vartheta>0$ and such that $a_n\ll_\varepsilon n^\varepsilon$ for all $\varepsilon>0$. Then, for any $\alpha,\beta\ll\frac1{\log T}$ and any $\eps>0$ as $T\to\infty$ we have
\begin{equation}
\begin{split}
&\int_{\R}A(\tfrac12+it)\zeta(\tfrac12+it+\alpha)\zeta(\tfrac12-it+\beta)\phi\Big(\frac{t}{T}\Big)\,dt=T\sum_{n\leq N}\frac{a_n}{n^{1+\alpha}}F_{\frac{\alpha+\beta}2}\Big(\frac{T}{2\pi n}\Big)+O_\eps(T^{\frac12+\eps}),
\end{split}\label{TwFormula}
\end{equation}
where, writing $\int_{(c)}\cdot \,dw:=\int_{c-i\infty}^{c+i\infty}\cdot\, dw$,  we have
\begin{equation*}%\label{dff}
 F_{\gamma}(x):=\frac{1}{2\pi i}\int_{(2)}\tilde\phi(s-\gamma)x^{s-\gamma-1} \zeta(s+\gamma)\zeta(s-\gamma)ds ,\qquad x>0, \ |\Re(\gamma)|<1.
 \end{equation*}

\end{prop}

Moving the line of integration to the left or to the right, one immediately sees that for any $C>0$ one has
\begin{equation}\label{fff}
F_\gamma(x)=\begin{cases}
\tilde \phi(1-2\gamma)\zeta(1-2\gamma)x^{-2\gamma}+\tilde\phi(1)\zeta(1+2\gamma)+O_C(x^{-C}) & \text{as $x\to+\infty$,}\\
%\tilde\phi(1)\log x+\tilde\phi(1)2\gamma+\tilde\phi'(1)+O_C(x^{-C}) & \text{as $x\to+\infty$,}\\
O_C(x^{C}) & \text{as $x\to0^+$,}
\end{cases}
\end{equation}
where the first line has to be interpreted as the limit if $\gamma=0$.
In particular, the sum on the right hand side of~\eqref{TwFormula} can be truncated at $T^{1+\delta}$ at negligible cost for any fixed $\delta>0$. 

If $N=O(T^{1-\delta})$ for fixed $\delta>0$ the asymptotic in Proposition~\ref{TwSM} is classical (see e.g.~\cite[Lemma 2.4]{FazzariThesis} or~\cite{BCHB,BCR} with straight-forward modifications), whereas the case where $a_n=0$ for $n<T^{1+\delta}$ is trivial. In particular, the new contribution of the above proposition is the handling of the terms $n\asymp T^{1+o(1)}$ in the transition between the two ranges of~\eqref{fff}, as needed when computing the lower order terms in Theorem~\ref{mt2}. 

\medskip

An asymptotic formula for the twisted fourth moment of zeta was computed in~\cite{HY} and~\cite{BBLR}. In particular, in the latter work the authors compute the asymptotic of $|\zeta|^4$ times a product of Dirichlet polynomials of lengths $T^{\vartheta_1}$ and $T^{\vartheta_2}$ with $\vartheta_1=\vartheta_2<\frac14$ and a  more precise bookkeeping in the proof would give $\vartheta_1+\vartheta_2+2\max(\vartheta_1,\vartheta_2)<1$ (and thus $\vartheta_1<\frac13$ if $\vartheta_2=0$). We refine the arguments of~\cite{BBLR}, and in fact slightly simplify the proof, so that to allow to handle the case $\vartheta_1+\vartheta_2<\frac12$. In doing so we also refine similarly the corresponding quadratic divisor problem, cf. Proposition~\ref{qdpt} below. %In particular this allows any $\vartheta_1<\frac12$ if $\vartheta_2=0$.

\begin{prop}\label{p4m}
Let $T\geq2, \vartheta_1,\vartheta_2\geq0$ and let $\phi:\R\to\R$ be of compact support in $\R_{>0}$ with derivatives satisfying $\Phi^{(j)}(x) \ll_j T^\epsilon$ for any $j\geq0$.
Let $A(s)=\sum_{a\leq T^{\vartheta_1}}\balpha_a a^{-s}$ and $B(s)=\sum_{b\leq T^{\vartheta_2}}\bbeta_b b^{-s}$ be Dirichlet polynomials with $\balpha_a \ll a^{\varepsilon}$ and $\bbeta_b \ll b^{\varepsilon}$. Then, for any $\alpha,\beta,\gamma,\delta\ll(\log T)^{-1}$ and any $\eps>0$ as $T\to\infty$ we have
\begin{align*}
&\int_\R \zeta (\tfrac12+it+\alpha)\zeta(\tfrac12+it+\beta)\zeta(\tfrac12-it+\gamma)\zeta(\tfrac12-it+\delta) A(\tfrac 12 + it) \overline{B(\tfrac 12 + it)}
\Phi\Big(\frac tT\Big)\,dt\\
&\hspace{3em}=  \sum_{g} \sum_{(a,b) = 1} \frac{\balpha_{ga} \overline{\bbeta_{gb}}}{g ab}\int_\R\Phi\Big(\frac tT\Big)\bigg ( \mathcal{Z}_{\alpha, \beta, \gamma, \delta,a,b} + \Big ( \frac{t}{2\pi} \Big )^{-\alpha - \beta - \gamma - \delta} \mathcal{Z}_{-\gamma, -\delta, -\alpha, -\beta,a,b} \bigg ) \,dt\\
&\hspace{3em}\quad + 
\sum_{g} \sum_{(a,b) = 1} \frac{\balpha_{ga} \overline{\bbeta_{gb}}}{g ab}\int_\R\Phi\Big(\frac tT\Big)\bigg ( \Big ( \frac{t}{2\pi} \Big )^{-\alpha - \gamma} \mathcal{Z}_{-\gamma, \beta,-\alpha, \delta,a,b}
+ \Big ( \frac{t}{2\pi} \Big )^{-\alpha - \delta} \mathcal{Z}_{-\delta, \beta, \gamma, -\alpha,a,b} \\ 
& \hspace{14em}
  + \Big ( \frac{t}{2\pi} \Big )^{-\beta - \gamma} \mathcal{Z}_{\alpha, -\gamma, -\beta, \delta,a,b}+ \Big ( \frac{t}{2\pi} \Big )^{-\beta - \delta} \mathcal{Z}_{\alpha, -\delta, \gamma, -\beta,a,b} \bigg)\,dt\\
&\hspace{3em}\quad+O_\eps\Big(T^{\frac12+\vartheta_1+\vartheta_2+\eps}+T^{\frac34+(\vartheta_1+\vartheta_2)/2+\eps}\Big),
\end{align*}
where
$\mathcal{Z}_{\alpha,\beta,\gamma,\delta,a,b}=\mathcal{A}_{\alpha,\beta,\gamma,\delta}\mathcal{B}_{\alpha,\beta,\gamma,\delta,a}\mathcal{B}_{\gamma,\delta,\alpha,\beta,b},$
with
\begin{align*}
\mathcal{A}_{\alpha, \beta, \gamma, \delta} =&   
\frac{\zeta(1 + \alpha + \gamma) \zeta(1 +\alpha + \delta) \zeta(1 + \beta + \gamma)\zeta(1+\beta+\delta)}{\zeta(2 + \alpha + \beta + \gamma + \delta)}, 
\end{align*}
$$
\mathcal{B}_{\alpha,\beta,\gamma,\delta,a}=\prod_{p^\nu||a}\left(\frac{\sum_{j=0}^{\infty}\sigma_{\alpha,\beta}(p^j)\sigma_{\gamma,\delta}(p^{j+\nu})p^{-j}}{\sum_{j=0}^{\infty}\sigma_{\alpha,\beta}(p^j)\sigma_{\gamma,\delta}(p^{j})p^{-j}}\right)
$$
and $\sigma_{\alpha,\beta}(n)=\sum_{n_1n_2=n}n_{1}^{-\alpha}n_{2}^{-\beta}$. 
\end{prop}

\medskip

The paper is organized as follows. In Section~\ref{S3} we deduce Theorem~\ref{mt2} from the above propositions. In Section~\ref{appl} we prove Theorem~\ref{applt}, whereas in Sections~\ref{S2} and~\ref{SL} we prove Propositions~\ref{TwSM} and~\ref{p4m}. Finally in the Appendix we prove that the kernels appearing in~\eqref{w1ld} match those in \cite{HughesThesis}.

\subsection*{Acknowledgments}
S. Bettin is member of the INdAM group GNAMPA and his work is partially supported by PRIN 2017 ``Geometric, algebraic and analytic methods in arithmetic''. A. Fazzari is supported by the FRG grant DMS 1854398.

%%%%%%%%%%%%%%%%%%%%%%%%%%%%%%%%%%%%%%%%%%%%%%%%%%%%%%%%%%%%%%%%%%%%%%%%%%%%

\section{The weighted one-level density}\label{S3}
\subsection{The case $k=1$}
We prove a shifted version of~\eqref{9febbr.1}. For $\alpha,\beta\ll 1/\log T$ we shall prove
\begin{equation}\label{fse}
\begin{split}
&\int_{\R}N_f(t)\zeta(\tfrac12+it+\alpha)\zeta(\tfrac12-it+\beta)
\phi\Big(\frac{t}{T}\Big)\, dt \\
&\hspace{6em}=\frac{T}{\log T} \int_{\R}{f}(x)
\bigg(\psi_{\alpha,\beta}(T)+\mathcal G_{\alpha,\beta}\Big(\frac{2\pi i x}{\log T}, T\Big) \bigg)dx+O_{\eps}(T^{\frac12+\eps}),
\end{split}
\end{equation}
where
\begin{equation}\label{bga}
\begin{split}
\psi_{\alpha,\beta}(T) &= \log(\tfrac T{2\pi})\Big(\tilde \phi(1-\alpha-\beta)\zeta(1-\alpha-\beta)(T/2\pi)^{-\alpha-\beta}+\tilde\phi(1)\zeta(1+\alpha+\beta)\Big)\\
&\quad+\tilde \phi'(1-\alpha-\beta)\zeta(1-\alpha-\beta)(T/2\pi)^{-\alpha-\beta}+\tilde\phi(1)'\zeta(1+\alpha+\beta),
\end{split}
\end{equation}
and
\begin{equation*}
\begin{split}
\mathcal  G_{\alpha,\beta}(y,T)&= \tilde\phi(1)\zeta(1+\alpha+\beta)\Big(\frac{\zeta'}{\zeta}(1+y+\alpha)+\frac{\zeta'}{\zeta}(1+y+\beta)\Big)\\
&\quad+\tilde\phi(1-\alpha-\beta)\Big(\frac{T}{2\pi}\Big)^{-\alpha-\beta}\zeta(1-\alpha-\beta) \Big(\frac{\zeta'}{\zeta}(1+y-\beta)+\frac{\zeta'}{\zeta}(1+y-\alpha)\Big)\\
&\quad- \tilde\phi(1-y-\alpha)\Big(\frac{T}{2\pi}\Big)^{-y-\alpha}\zeta(1-y-\alpha)\zeta(1-y+\beta)\\
&\quad-\tilde\phi(1-y-\beta)\Big(\frac{T}{2\pi}\Big)^{-y-\beta}\zeta(1-y+\alpha)\zeta(1-y-\beta).
\end{split}
\end{equation*}
Theorem~\ref{mt2} then follows by letting $\alpha,\beta\to0$.

\medskip

First, we need the following version of the explicit formula.

\begin{lemma}[\cite{HR}, Lemma 2.1]\label{explicit formula}
Let $g$ be a smooth, compactly supported function and $h(r)=\int_{-\infty}^{+\infty}g(u)e^{iru}du$. Moreover we set $\Omega(r)=\frac{1}{2}\Psi(\frac{1}{4}+\frac{1}{2}ir)+\frac{1}{2}\Psi(\frac{1}{4}-\frac{1}{2}ir)-\log\pi$, where $\Psi(s)=\frac{\Gamma'}{\Gamma}(s)$ is the polygamma function, and we denote by $\Lambda(n)$ the von Mangoldt function. Then
\begin{equation}\begin{split}\notag \sum_\gamma h(\gamma)=\frac{1}{2\pi}\int_{-\infty}^{+\infty}h(r)\Omega(r)dr-\sum_{n=1}^{\infty}\frac{\Lambda(n)}{\sqrt n}\Big(g(\log n)&+g(-\log n)\Big)+h\Big(-\frac{i}{2}\Big)+h\Big(\frac{i}{2}\Big). \end{split}\end{equation}
\end{lemma}

We let $f$ be as in Theorem~\ref{mt2} and apply Lemma~\ref{explicit formula} with $h(r)=f(\frac{V \log T }{2\pi}(r-t))$. We obtain
\begin{equation}\label{bft}
N_f(t)={N^*_f(t)}+S_f(t), 
\end{equation}
where
\begin{align*}
{N^*_f(t)}&=\frac{1}{2\pi}\int_{\R}f\bigg(\frac{ \log T}{2\pi}(r-t)\bigg)\Omega(r)dr+f\Big(-\frac{ \log T}{2\pi}(t+\tfrac{i}{2})\Big)+f\Big(\frac{\log T}{2\pi}(\tfrac{i}{2}-t)\Big)%\notag
%\label{defNbarra}
,\\
S_f(t)&=-\frac{1}{\log T}\sum_{n=1}^{\infty}\frac{\Lambda(n)}{\sqrt n}\hat{f}\left(\frac{\log n}{\log T}\right)(n^{-it}+n^{it}).%\label{defS}
\end{align*}

Now, we observe that by Stirling's formula we have $\Omega(t+u)=\log(\frac t{2\pi})+O((1+|t|)^{-1/2})$ for $u,t\in\R$ with $u=O(|t|^{1/2})$. Thus, since $\Omega(r)=O(\log(1+|r|))$, by the fast decaying of $f$ we obtain the following slight strengthening of~\cite[Lemma 2.2]{HR}
\begin{align*}
\frac1{2\pi}\int_{\R}f\Big(\frac{\log T}{2\pi}(r-t)\Big)\log(\tfrac t{2\pi})dr=\frac{\log(t/2\pi)}{\log T}\hat{f}(0)+O((1+|t|)^{-1/2}).
%&=\frac1{2\pi}\int_{\R}f\bigg(\frac{\log T}{2\pi}r\bigg)\Omega(r+t)dr\\
%&=\frac1{2\pi}\int_{-\sqrt {|t|}}^{\sqrt {|t|}}f\bigg(\frac{\log T}{2\pi}r\bigg)(\Omega(t)+O((1+|t|)^{-1/2}))\,dr\\
%&\quad+\int_{|r|>\sqrt {|t|}}f\bigg(\frac{\log T}{2\pi}r\bigg)O(\log(2+|r|))\,dr=\dots
\end{align*}
Moreover, by Fourier inversion and integrating by parts $m$-times, we get
\begin{equation}\notag
f(z)=\left(\frac{i}{2\pi}\right)^m\int_{-\infty}^{+\infty}\frac{e^{2\pi izy}}{z^m}\hat f^{(m)}(y)dy=O_m(|z|^{-m}e^{2\pi a|\Im(z)|})
\end{equation}
for $z\in\C_{\neq0}$ and $\supp(\hat f)\subseteq [-a,a]$. Thus, choosing $m$ sufficiently large with respect to $a$, we obtain
\begin{align}\label{nstf}
{N^*_f(t)}=\frac{\log(t/2\pi)}{\log T}\hat{f}(0)+O(T^{-1/2}),
\end{align}
for $|t|\asymp T$. In particular, letting $\mathscr Z_{\alpha,\beta}(t):=\zeta(\frac12+it+\alpha)\zeta(\frac12-it+\beta)$, we have
\begin{equation*}%\label{nfs}
\int_\R N^*_f(t)\mathscr Z_{\alpha,\beta}(t)\phi\Big(\frac{t}{T}\Big)\, dt 
=\frac{\hat{f}(0)}{\log T}\int_\R
\mathscr Z_{\alpha,\beta}(t)\phi\Big(\frac{t}{T}\Big)\log(\tfrac t{2\pi})\, dt +O(T^{1/2}\log T).
\end{equation*}
%By Stirling approximation $\frac{\Gamma'}{\Gamma}(s)=\log s+O(1/|s|)$ and thus we can replace $\Omega(t)$ by  $\log \tfrac{t}{2\pi}$ at a negligible cost.
Then, an application of Proposition~\ref{TwSM} with $A(s)=1$ and $\phi(\cdot )(\log (T/2\pi)+\log(\cdot))$ in place of $\phi(\cdot)$
yields
\begin{equation}\label{nfs2}
\begin{split}
\int_\R N^*_f(t)\mathscr Z_{\alpha,\beta}(t) \phi\Big(\frac{t}{T}\Big)\, dt&=\frac{\hat{f}(0)}{\log T}\psi_{\alpha,\beta}(T)
+ O_\eps(T^{\frac12+\varepsilon})
\end{split}
 \end{equation}
with $\psi_{\alpha,\beta}(T)$ as in~\eqref{bga}.
\medskip

We now move to computing the weighted average of $S_f$. By Proposition~\ref{TwSM}, for any $\eps>0$ we have 
$$ \int_{\R}S_f(t)\mathscr Z_{\alpha,\beta}(t)\phi\Big(\frac{t}{T}\Big)\, dt = -\frac{T}{\log T}\sum_n\frac{\Lambda(n)}{n}\hat f\Big(\frac{\log n}{\log T}\Big)(n^{-\alpha}+n^{-\beta}) F_{\frac{\alpha+\beta}2}\left(\frac{T}{2\pi n}\right) + O_\eps(T^{\frac12+\eps}).$$
We let $\gamma:=\frac{\alpha+\beta}2$ and expand the Fourier transform, so that the main term on the right becomes
\begin{align*}
& \frac{T}{\log T}\int_{\R}f(x)\frac{1}{2\pi i}\int_{(2)}\tilde\phi(s-\gamma)\left(\frac{T}{2\pi}\right)^{s-\gamma-1}\zeta(s-\gamma)\zeta(s+\gamma)\\
&\hspace{12em}\times \bigg(\frac{\zeta'}{\zeta}\Big(s+\frac{2\pi ix}{\log T}+\frac{\alpha-\beta}2\Big)+\frac{\zeta'}{\zeta}\Big(s+\frac{2\pi ix}{\log T}-\frac{\alpha-\beta}2\Big)\bigg)\,ds\,dx.
 \end{align*}
We shift the integral over $s$ to $\Re(s)=0$. The integral on the new line of integration is $O(1)$, whereas the residues at the poles at $s=1-\frac{2\pi i x}{\log T}\pm\frac{\alpha-\beta}2$ and $s=1\pm\gamma $ give a contribution of 
\begin{align*}
&- \frac{T}{\log T}\int_{\R}f(x)\bigg(\tilde\phi\Big(1-\frac{2\pi i x}{\log T}-\alpha\Big)\Big(\frac{T}{2\pi}\Big)^{-\frac{2\pi i x}{\log T}-\alpha}\zeta\Big(1-\frac{2\pi i x}{\log T}-\alpha\Big)\zeta\Big(1-\frac{2\pi i x}{\log T}+\beta\Big)\\
&\hspace{8em}+\tilde\phi\Big(1-\frac{2\pi i x}{\log T}-\beta\Big)\Big(\frac{T}{2\pi}\Big)^{-\frac{2\pi i x}{\log T}-\beta}\zeta\Big(1-\frac{2\pi i x}{\log T}+\alpha\Big)\zeta\Big(1-\frac{2\pi i x}{\log T}-\beta\Big)\bigg)\,dx
 \end{align*}
and
\begin{align*}
& \frac{T}{\log T}\int_{\R}f(x)\bigg(\tilde\phi(1)\zeta(1+2\gamma)\bigg(\frac{\zeta'}{\zeta}\Big(1+\frac{2\pi ix}{\log T}+\alpha\Big)+\frac{\zeta'}{\zeta}\Big(1+\frac{2\pi ix}{\log T}+\beta\Big)\bigg)\\
&\hspace{6em}+\tilde\phi(1-2\gamma)\Big(\frac{T}{2\pi}\Big)^{-2\gamma}\zeta(1-2\gamma) \Big(\frac{\zeta'}{\zeta}\Big(1+\frac{2\pi ix}{\log T}-\beta\Big)+\frac{\zeta'}{\zeta}\Big(1+\frac{2\pi ix}{\log T}-\alpha\Big)\Big)\bigg)\,dx.
 \end{align*}
Thus, 
\begin{align*}%\label{eqst}
\int_{\R}S_f(t)\mathscr Z_{\alpha,\beta}(t)\phi\Big(\frac{t}{T}\Big)\, dt&= \frac{T}{\log T}\int_{\R}{f}(x)
\, \mathcal G_{\alpha,\beta}\Big(\frac{2\pi i x}{\log T}, T\Big) \,dx+O_\eps(T^{\frac12+\eps})
\end{align*}
and~\eqref{fse} follows by~\eqref{bft} and~\eqref{nfs2}.

\subsection{The case $k=2$}
As in the previous section one easily sees that 
\begin{equation}\label{29apr.1}
\frac{1}{c_4T(\log T)^4}\int_{\mathbb R}N_f^*(t)|\zeta(\tfrac{1}{2}+it)|^{4}\phi\Big(\frac{t}{T}\Big) dt 
= \tilde\phi(1) \hat f(0)+O(1/\log T).
\end{equation}
To compute the average of $S_f(t)$ it is convenient to introduce shifts $\alpha,\beta,\gamma,\delta$ of size $\ll 1/\log T$ and consider
\begin{equation*}
I_{\alpha,\beta,\gamma,\delta}:=\int_{\mathbb R} P(t)\zeta(\tfrac{1}{2}+\alpha+it)
\zeta(\tfrac{1}{2}+\beta+it)\zeta(\tfrac{1}{2}+\gamma-it)\zeta(\tfrac{1}{2}+
\delta-it)\phi\Big(\frac{t}{T}\Big)dt. 
\end{equation*}
with $P(t):=\sum_{n\geq1}\frac{\Lambda(n)}{n^{1/2+it}}\hat f(\frac{\log n}{\log T})$.
By Proposition~\ref{p4m} we have
\begin{equation*}\begin{split}
I_{\alpha,\beta,\gamma,\delta}&=\sum_{n=1}^{\infty}\frac{\Lambda(n)}{n}\hat f\Big(\frac{\log n}{\log T}\Big)\int_{\mathbb R}\Big( 
\mathcal Z_{\alpha,\beta,\gamma,\delta,n,1} 
+ (\tfrac{t}{2\pi})^{-\alpha-\gamma}\mathcal Z_{-\gamma,\beta,-\alpha,\delta,n,1}  
\\
&\hspace{5.5em}+ (\tfrac{t}{2\pi})^{-\alpha-\delta}\mathcal Z_{-\delta,\beta,\gamma,-\alpha,n,1}+ (\tfrac{t}{2\pi})^{-\beta-\gamma}\mathcal Z_{\alpha,-\gamma,-\beta,\delta,n,1}    
+ (\tfrac{t}{2\pi})^{-\beta-\delta}\mathcal Z_{\alpha,-\delta,\gamma,-\delta,n,1}  \\
&\hspace{5.5em}+
 (\tfrac{t}{2\pi})^{-\alpha-\beta-\gamma-\delta}\mathcal Z_{-\gamma,-\delta,-\alpha,-\beta,n,1} 
\Big) \phi\Big(\frac{t}{T}\Big)dt.
\end{split}\end{equation*}
For $n=p^\nu \ll T$ a prime power, we have
\begin{equation*}\begin{split}\notag
\mathcal B_{\alpha,\beta,\gamma,\delta,p^\nu} 
& =\bigg(\sum_{j=0}^\infty \sigma_{\alpha,\beta}(p^j)\sigma_{\gamma,\delta}(p^{j+\nu})p^{-j}\bigg)
\bigg(\sum_{j=0}^\infty \sigma_{\alpha,\beta}(p^j)\sigma_{\gamma,\delta}(p^{j})p^{-j}\bigg)^{-1} \\ 
& = \big(\sigma_{\gamma,\delta}(p^{\nu}) + O(\nu/p)\big)(1+O(1/p))^{-1}
= \sigma_{\gamma,\delta}(p^{\nu}) + O(\nu/p)
\end{split}\end{equation*}
and in particular  $\mathcal B_{\alpha,\beta,\gamma,\delta,p} = p^{-\gamma} + p^{-\delta}  +O(1/p)$ if $\nu=1$. Estimating trivially the error term and the terms with $\nu\geq2$ we obtain
\begin{equation*}\begin{split}%\label{13apr.0} 
I_{\alpha,\beta,\gamma,\delta}&= \sum_{p=1}^{\infty}\frac{\log p}{p}\hat f\Big(\frac{\log p}{\log T}\Big)\int_{\mathbb R}\Big( 
\mathcal A_{\alpha,\beta,\gamma,\delta} \big(p^{-\gamma} + p^{-\delta}\big)
+ (\tfrac{t}{2\pi})^{-\alpha-\gamma}\mathcal A_{-\gamma,\beta,-\alpha,\delta}  \big(p^{\alpha} + p^{-\delta}\big) \\
&\quad+ (\tfrac{t}{2\pi})^{-\alpha-\delta}\mathcal A_{-\delta,\beta,\gamma,-\alpha} \big(p^{-\gamma} + p^{\alpha}\big) 
+ (\tfrac{t}{2\pi})^{-\beta-\gamma}\mathcal A_{\alpha,-\gamma,-\beta,\delta}  \big(p^{\beta} + p^{-\delta}\big)   \\ 
&\quad+ (\tfrac{t}{2\pi})^{-\beta-\delta}\mathcal A_{\alpha,-\delta,\gamma,-\delta}  \big(p^{-\gamma} + p^{\beta}\big)  
+ (\tfrac{t}{2\pi})^{-\alpha-\beta-\gamma-\delta}\mathcal A_{-\gamma,-\delta,-\alpha,-\beta}  \big(p^{\alpha} + p^{\beta}\big) 
\Big) \phi\Big(\frac{t}{T}\Big)dt\\
&\quad+O(T(\log T)^4).
\end{split}\end{equation*}
  Now for any $z\ll1/\log T$ we have
\begin{equation*}%\label{13apr.1} 
\sum_p \frac{\log p}{p}\hat f\left(\frac{\log p}{\log T}\right) p^{-z}
= \log T \int_0^{\infty} \hat f(y) T^{-yz} dy+O(1)
\end{equation*}
and thus 
\begin{equation}\begin{split}\notag%\label{13apr.2} 
I_{\alpha,\beta,\gamma,\delta}&= \log T \int_0^\infty \hat f(y)
\int_{\mathbb R}\Big( 
\mathcal A_{\alpha,\beta,\gamma,\delta} \big(T^{-y\gamma} + T^{-y\delta}\big)
+ (\tfrac{t}{2\pi})^{-\alpha-\gamma}\mathcal A_{-\gamma,\beta,-\alpha,\delta}  \big(T^{y\alpha} + T^{-y\delta}\big) \\
&\quad+ (\tfrac{t}{2\pi})^{-\alpha-\delta}\mathcal A_{-\delta,\beta,\gamma,-\alpha} \big(T^{-y\gamma} + T^{y\alpha}\big) 
+ (\tfrac{t}{2\pi})^{-\beta-\gamma}\mathcal A_{\alpha,-\gamma,-\beta,\delta}  \big(T^{y\beta} + T^{-y\delta}\big)   \\ 
&\quad+ (\tfrac{t}{2\pi})^{-\beta-\delta}\mathcal A_{\alpha,-\delta,\gamma,-\delta}  \big(T^{-y\gamma} + T^{y\beta}\big)  
+ (\tfrac{t}{2\pi})^{-\alpha-\beta-\gamma-\delta}\mathcal A_{-\gamma,-\delta,-\alpha,-\beta}  \big(T^{y\alpha} + T^{y\beta}\big) 
\Big) \phi\Big(\frac{t}{T}\Big)dt \;dy\\
&\quad+O(T(\log T)^4).
\end{split}\end{equation}
Computing the limit as $\alpha,\beta,\gamma,\delta\to 0$ with the help of Sage, we obtain
\begin{equation*}
\frac{1}{c_4T(\log T)^4}\int_{\mathbb R}P(t)|\zeta(\tfrac{1}{2}+it)|^{4}\phi\Big(\frac{t}{T}\Big) dt 
= \tilde\phi(1)\log T \int_{0}^{\infty}\hat f(y)(2y^3-4y+2)dy + O(T).
\end{equation*}
Since $S_f(t)=-\frac1{\log T}(P(t)+P(-t))$ and $f$ is even, we obtain by conjugation that 
\begin{equation*}%\label{29apr.2}
\frac{1}{c_4T(\log T)^4}\int_{\mathbb R}S_f(t)|\zeta(\tfrac{1}{2}+it)|^{4}\phi\Big(\frac{t}{T}\Big) dt 
= - 2 \tilde\phi(1)  \int_{0}^{\infty}\hat f(y)(2y^3-4y+2)dy + O_{f,\phi}\Big(\frac{1}{\log T}\Big).
\end{equation*}
By~\eqref{29apr.1} we deduce
\begin{equation*}\begin{split}%\label{29apr.3}
&\frac{1}{c_4 T(\log T)^4}\int_{\mathbb R}N_f(t)|\zeta(\tfrac{1}{2}+it)|^{4}\phi\Big(\frac{t}{T}\Big) dt \\
&\hspace{2cm}= 
\tilde\phi(1) \int_{-\infty}^{+\infty}\hat f(y) \Big(\delta_0(y) +(-2|y|^3+4|y|-2)\chi_{[-1,1]}(y)  \Big)dy +  O_{f,\phi}\Big(\frac{1}{\log T}\Big)
\end{split}\end{equation*}
where we could add $\chi_{[-1,1]}$ in the above expression since $\hat f$ is supported in $(-\frac{1}{2},\frac{1}{2})$. The claim then follows by Plancherel identity. %Being
%$$\delta_0(y) +(-2|y|^3+4|y|-2)\chi_{[-1,1]}(y)$$
%the Fourier transform of
%$$1- \frac{2+\cos(2\pi x)}{(\pi x)^2} + \frac{3\sin(2\pi x)}{(\pi x)^3}+\frac{3(\cos(2\pi x)-1)}{2(\pi x)^4}.$$

\section{Zeros next to large values}\label{appl}

We assume the Riemann hypothesis throughout this section. Also, we denote $L:=\log\log T$ and $\chi:=\chi_{[-2\pi,2\pi]}$. Finally, let $\eta$ be any fixed real valued function, which is even, smooth, with Fourier support in $(-1,1)$ and which satisfies $\chi\leq \eta$. For any $\delta>0$, let then $\eta_\delta(t):=\eta(t \delta)$ so that $\hat \eta_\delta(x)=\delta^{-1}\hat \eta(x/\delta)$ is supported in $(-\delta,\delta)$.

\begin{lemma}\label{bfn}
For any $B>0$ if $A$ is sufficiently large we have 
$$\meas\bigg(\bigg\{t\in[T,2T]\biggm | N_{\chi}(t)>  A\frac{L}{\log L}\bigg\}\bigg)=O(T/\log^B T).$$
\end{lemma}
\begin{proof}
For any $Q\in\N_{>2\pi}$, we have $\chi\leq \chi_{[-Q,Q]}\leq g:=\eta_{1/Q}$ and thus $N_{\chi}\leq N_g$. We split $N_g$ as in~\eqref{bft}. By~\eqref{nstf} we have ${N^*_g(t)}=O(Q)$. Moreover, bounding the contribution of prime powers by Mertens' theorem, we deduce that $N_g(t)=S^{*}_g(t)+O(Q)$, where
\begin{align*}
S^*_g(t)&:=-\frac{Q}{\log T}\sum_{p\geq2}\frac{\log p}{\sqrt p}\hat{\eta}\left(\frac{Q\log p}{\log T}\right)(p^{-it}+p^{it}),%\\
%S^{**}_g(t)&:=-\frac{Q}{\log T}\sum_{p^m,\,m\geq2}^{\infty}\frac{\log p}{ p^{m/2}}\hat{\eta}\left(\frac{mQ\log p}{\log T}\right)(p^{-it}+p^{it}).
\end{align*}
By~\cite[Lemma~3]{Sou} if $k<Q/{2}$ we have
\begin{align*}
\int_{T}^{2T}|S_g^*(t)|^{2k}\,dt &\ll Tk! \frac{Q^{2k}}{(\log T)^{2k}} \bigg( \sum_{p\geq 2}^{\infty}\frac{(\log p)^2}{ p}\bigg|\hat{\eta}\left(\frac{Q\log p}{\log T}\right)\bigg|^2 \bigg)^k\\
& \ll Tk! \frac{Q^{2k}}{(\log T)^{2k}} (C \log T^{1/Q})^{2k}
 \ll  T k!C^{2k}
\end{align*}
for some constant $C$. 
Then, for any $k\geq3$, taking $Q=[3k]$ and $A>0$ we deduce
\begin{align*}
\meas(\{t\in[T,2T]\mid S_g^*(t)>   \sqrt C k/\log k\})&\leq \int_{T}^{2T}\frac {|S_g^*(t)|^{2k}}{ (\sqrt C k/\log k)^{2k}}\,dt \\
&\ll  T k! (k/\log k)^{-2k}\ll \sqrt k e^{-k(1+\log k-2\log\log  k)},
\end{align*}
by Stirling's formula.  The result then follows by taking $k=A\log\log T/\log\log\log T$ with $A$ sufficiently large.
\end{proof}

For $a>0$, $u\in\R$, let $I_a:=[- a/\log T, a/\log T]$ and $M_a(u):=\max_{t\in u+I_a}|\zeta(\tfrac12+it)|$. Also, denote $I=I_1$ and $M=M_1$ and recall that $Z(T):=\{\gamma\in[T,2T]\mid \zeta(\frac12+i\gamma)=0\}$.
\begin{lemma}\label{fl}
Assume that Conjecture~\ref{smtco} holds for some $k\in\N$ and all even functions $f$ with Fourier support in $(-\delta,\delta)$ for some $\delta>0$.
Let $V>0$. Then $\#\{\gamma\in Z(T)\mid  M(\gamma) > V\}=O_k(T(\log T)^{k^2+1}/V^{2k})$.
\end{lemma}
\begin{proof}
We have
\begin{align*}
\sum_{\gamma\in Z(T), M(\gamma)>V}1&=\sum_{\gamma\in Z(T), M(\gamma) >V }\int_{t\in\gamma+I}2\log T\,dt \leq \int_{T-1}^{2T+1} \sum_{\gamma\in t+I, M(\gamma) >V}2\log T\,dt.
\end{align*}
Now, if $\gamma\in t+I$ then $M_2(t)>M(\gamma)$. In particular, using Rankin's trick we obtain
\begin{align*}
\sum_{\gamma\in Z(T), M(\gamma)>V}1 \leq \frac{2\log T}{V^{2k}}\int_{T-1}^{2T+1} \sum_{\gamma\in Z(T)\cap(t+I)}M_2(t)^{2k}\,dt\leq\frac{2\log T}{V^{2k}}\int_{T-1}^{2T+1} M_2(t)^{2k} N_{\chi}(t)\,dt.
\end{align*}
By~\cite[Equation~(6)]{ABBRS} we have
\begin{align*}%\label{mda}
M_2(t)^{2k}&\leq \frac{|\zeta(\frac12+it-\frac{ i}{\log T})|^{2k}+|\zeta(\frac12+it+\frac{ i}{\log T})|^{2k}}{2}\\
&\quad+\int_{-1/\log T}^{1/\log T}| k\zeta(\tfrac12+it+iu)^{2k-1}\zeta'(\tfrac12+it+iu)|\,du.
\end{align*}
By Conjecture~\ref{smtco} we have
\begin{align*}
\int_{T-1}^{2T+1}| \zeta(\tfrac12+it\pm  i/\log T)|^{2k} N_{\chi}(t)\,dt&\leq \int_{T-2}^{2T+2}| \zeta(\tfrac12+ it)|^2 N_{g}(t)\,dt\\
&=O(T(\log T)^{k^2}).
\end{align*}
Also, by~\cite[Lemma~1]{Rad} for $|u|\leq1/\log T$ we have
\begin{align*}
\zeta'(\tfrac12+it+iu)&=\frac1{2\pi i }\int_{|w|=\frac2{\log T}} \frac{\zeta(\tfrac12+it+iu+w)}{w^2}\,dw\\
&\ll (\log T)^2 \int_{|w|=\frac2{\log T}}|\zeta(\tfrac12+it+iu+\Re(w))|\,dw.
\end{align*}
Thus, if we let let $g=\eta_{\min(1,\delta)/2}$ (so that $g$ has Fourier support in $(-\delta,\delta)$ and satisfies $\chi\leq g$) we have 
\begin{align*}
&\int_{T-1}^{2T+1} N_{\chi}(t) \int_{-1/\log T}^{1/\log T}| k\zeta(\tfrac12+it-iu)^{2k-1}\zeta'(\tfrac12+it+iu)|\,du\,dt\\
&\quad\ll (\log T)^2 \int_{|w|=\frac2{\log T}}\int_{-1/\log T}^{1/\log T}\int_{T-1}^{2T+1}|\zeta(\tfrac12+it-iu)^{2k-1}\zeta(\tfrac12+it+iu+\Re(w))|N_g(t)\,dt\,du\,dw\\
&\quad\ll T(\log T)^{k^2}
\end{align*}
by H\"older's inequality and Conjecture~\ref{smtco}. Thus $\int_{T-1}^{2T+1} M_2(t)^{2k} N_{\chi}(t)\,dt\ll T(\log T)^{k^2}$ and the lemma follows.
\end{proof}

\begin{lemma}\label{sl}
Assume that Conjecture~\ref{smtco} holds for some $k\in\N$ and all even functions $f$ with Fourier support in $(-\delta,\delta)$ for some $\delta>0$. 
Let $U\geq \frac92 k^3 L/\sqrt{\log L} $ for some large enough $C>0$. Then
$$\#\{\gamma\in Z(T)\mid  M(\gamma) > e^{-U } (\log T)^k \}\gg T (\log T)^{1-k^2}e^{-2k U}.$$
\end{lemma}
\begin{proof}
For $u\in\R$, we let $J_u:=\{t\in[T,2T]\mid |\zeta(\tfrac12+it)|> e^{u+ kL}\}$. Also, we let $J=J_{-U}\cap J_{U}^{c}$, where %$U:=\sqrt {  L\log L }(1-1/\log L)$
the complement is taken in $[T,2T].$
We have
\begin{align*}
\sum_{\gamma\in Z(T), M(\gamma)>e^{kL-U}}1&\geq  \sum_{\gamma\in Z(T), M(\gamma) >e^{kL-U} }\int_{(\gamma+I)\cap J}\frac{\log T}2\,dt = \int_{J} \frac{\log T}2 \sum_{\gamma\in t+I, M(\gamma) >e^{kL-U}}1\,dt.
\end{align*}
The condition $M(\gamma)>e^{kL-U}$ in the sum over the zeros can  be dropped since it is implied by the remaining conditions and the sum is simply $N_{\chi}(t)$. Thus,
\begin{align}\label{fad}
\sum_{\gamma\in Z(T), M(\gamma)>e^{kL-U}}1&\geq \int_{J} N_{\chi}(t) \frac{\log T}2\,dt\geq \frac{\log T}{2e^{2k(kL+U)}}\int_{J} |\zeta(\tfrac12+it)|^{2k} N_{\chi}(t) \,dt.
\end{align}
We wish to show that we can extend the integral to the full domain $[T,2T]$. We let $Y:=\{t\in[T,2T]\mid N_{\chi}(t)>  A\frac{L}{\log L}\}$ where we fix a sufficiently large $A$ so that $\meas(Y)=O(T/(\log T)^{8k^2})$, by Lemma~\ref{bfn}. Then, 
\begin{align}\label{exin}
\int_{[T,2T]\setminus J}& |\zeta(\tfrac12+it)|^{2k} N_{\chi}(t) \,dt\ll
\int_{Y} |\zeta(\tfrac12+it)|^{2k} N_{\chi}(t) \,dt+\frac{L}{\log L}\int_{[T,2T]\setminus J} |\zeta(\tfrac12+it)|^{2k}  \,dt.
\end{align}
By the H\"older inequality the first integral on the right is bounded by
\begin{align*}
 \bigg(\int_{T}^{2T} |\zeta(\tfrac12+it)|^{4k}\,dt\bigg)^{1/2}\bigg(\int_{T}^{2T} N_{\chi}(t)^4\,dt\bigg)^{1/4}\meas(Y)^{1/4}\ll T,
\end{align*}
since the bound $\int_{T}^{2T} N_{\chi}(t)^4\,dt\ll T$ is implicit in the proof Lemma~\ref{bfn}. The second integral on the right of~\eqref{exin} can be bounded as in~\cite[Proof of Corollary 1.2]{AB}. More specifically, we observe that the contribution of the integral over the domain $J_{-2kL/3}^c$ is trivially $O(T(\log T)^{2k^2/3})$. Also, by~\cite{Sou} we have $\meas(J_u)\ll TL^{-1/2}e^{-(u+kL)^2/L+19 k^3 L/\log L}$ for $|u|\leq \frac34 kL$ and thus applying the Cauchy-Schwarz inequality, the contribution of the integral over $J_{2kL/3}$ is $O(T(\log T)^{2k^2/3})$. Thus,
\begin{align*}
\frac1{T(\log T)^{k^2}}\int_{[T,2T]\setminus J} |\zeta(\tfrac12+it)|^{2k}  \,dt&\leq \sum_{u\in\Z\atop U-1\leq |u| \leq \frac23kL+1}e^{2ku+k^2L}\meas(J_{u-1}\setminus J_{u})+O((\log T)^{-k^2/3})\\
&\ll \sum_{|u| \geq U-1} L^{-1/2}e^{-u^2/L+19 k^3 L/\log L}+O((\log T)^{-k^2/3})\\
&\ll \sqrt{L/U^2}  e^{-U^2/L+19 k^3 L/\log L}+ O((\log T)^{-k^2/3})=o(\log L/L)
\end{align*}
for $U\geq \frac92  k^3 L/\sqrt{\log L} $. Thus by~\eqref{exin} and Conjecture~\ref{smtco} we have
\begin{align*}
\int_{J} |\zeta(\tfrac12+it)|^2 N_{\chi}(t) \,dt=\int_{T}^{2T} |\zeta(\tfrac12+it)|^2 N_{\chi}(t) \,dt+o(T(\log T)^{k^2})\gg T(\log T)^{k^2}.
\end{align*}
The result then follows by~\eqref{fad}.
\end{proof}
\begin{remark}\label{dad}
By~\cite{AB}, for $|u|\leq \frac34 L$, $k=1$, we have the stronger bound $\meas(J_u)\ll L^{-1/2}e^{-(u+L)^2/L}$. In particular, for $k=1$ the same proof gives the same condition under the weaker hypothesis $U\geq \sqrt{L\log L}(1-1/\log L)$ which leads to the stronger result stated after Corollary~\ref{appltc}.
\end{remark}

\begin{corol}
Assume that Conjecture~\ref{smtco} holds for some $k\in\N$ and all even functions $f$ with Fourier support in $(-\delta,\delta)$ for some $\delta>0$. Then, for  $U\geq 5 k^3 L/\sqrt{\log L} $   we have
$$e^{-2kU}\ll \frac{\#\{\gamma\in Z(T)\mid  -U< \log(M(\gamma))-kL<U\}}{T(\log T)^{1-k^2}}\ll e^{2kU}.$$
\end{corol}
\begin{proof}
For $0\leq U_1,U_2\leq \infty$, let $S(U_1,U_2):=\#\{\gamma\in Z(T)\mid  -U_1< \log(M(\gamma))-kL<U_2\}$. Also, let $V= \frac92 k^3 L/\sqrt{\log L} $. Then, we have 
\begin{align*}
S(U,U)\geq S(V,U)&=S(V,\infty )+O(Te^{(1-k^2)L-2kU})\gg T (\log T)^{1-k^2}e^{-2k V} 
\end{align*} 
 by Lemma~\ref{fl} and Lemma~\ref{sl}. Finally, $S(U,U)\leq S(\infty,U)\ll Te^{(1-k^2)L+2kU}$ by Lemma~\ref{fl}.
\end{proof}

\section{The twisted second moment}\label{S2}
In this section we prove Proposition~\ref{TwSM}.  For $n\in\N$ and $\alpha,\beta\in\C$ with $|\alpha|,|\beta|\ll 1/\log T$, let 
\begin{equation}\label{5genn.0}
I_n:=\int_{\R} n^{it}\zeta(1/2+it+\alpha)\zeta(1/2-it+\beta)\phi\Big(\frac{t}{T}\Big)\,dt.
\end{equation}
Also, let  $G_{\alpha,\beta}(w,t):=e^{w^2}\frac{(\frac14-(w+it+\alpha)^2)(\frac14-(w-it+\beta)^2)}{(w+it+\alpha)^2(w-it+\beta)^2}$
and fix a small $\delta>0$.

First of all we note that if $n$ is substantially larger than $T$, say $n>T^{1+2\delta}$, then $I_n=O(T^{-C})=F_\phi(\frac{T}{2\pi n})+O(T^{-C})$ for any $C>0$. Indeed, for $t\asymp T$ we can write (see \cite[Theorem 5.3]{I-K}) 
\begin{align}\label{afe1}
\zeta(\tfrac12+it+\alpha)\zeta(\tfrac12-it+\beta)
&=\mathcal S_{\alpha,\beta}(t)+\mathcal X_{\alpha,\beta}(t)\mathcal S_{-\beta,-\alpha}(t)+O(T^{-C-1}),
\end{align}
where
\begin{align}\label{slad}
\mathcal S_{\alpha,\beta}(t)
&=\sum_{m_1m_2<T^{1+\delta}}\frac{1}{2\pi i}\int_{(3/2)} \frac{(m_2/m_1)^{it}m_1^{-\alpha}m_2^{-\beta}g_{\alpha,\beta}(w,t)G_{\alpha,\beta}(w)}{\sqrt{m_1m_2}(\pi m_1m_2)^w }\frac{dw}{w}
\end{align}
and
\begin{align*}
g_{\alpha,\beta}(w,t):=\frac{\Gamma\big(\frac14+\frac{w+it+\alpha}{2}\big)\Gamma\big(\frac14+\frac{w-it+\beta}{2}\big)}{\Gamma\big(\frac14+\frac{it+\alpha}{2}\big)\Gamma\big(\frac14+\frac{-it+\beta}{2}\big)},\qquad \mathcal X_{\alpha,\beta}(t):=\frac{\Gamma\big(\frac14-\frac{it+\alpha}{2}\big)\Gamma\big(\frac14-\frac{-it+\beta}{2}\big)}{\Gamma\big(\frac14+\frac{it+\alpha}{2}\big)\Gamma\big(\frac14+\frac{-it+\beta}{2}\big)}\pi^{\alpha+\beta}.
\end{align*}
We insert~\eqref{afe1}-\eqref{slad} into~\eqref{5genn.0} and repeatedly integrate by parts with respect to $t$. We then deduce $I_n=O(T^{-C})$ for any $C>0$ since  $|\log (\frac{m_2n}{m_1})|\gg \log T>1$ and $\partial_{t}^{(j)}g_{\alpha,\beta}(w,t)\ll_j |t|^{\Re(w)-j}(1+|w|^{j+1})$ for $\Re(w)\ll1$, by Stirling's formula. 

As in~\cite[Lemma 3]{LiRadz} we approximate $G_{\alpha,\beta}$ and $g_{\alpha,\beta}$ at first order simplifying $\mathcal S_{\alpha,\beta}(t)$ to
\begin{equation}\notag
\mathcal S_{\alpha,\beta}(t)=\sum_{m_1m_2<T^{1+\delta}}\frac{m_1^{-\alpha}m_2^{-\beta}}{\sqrt{m_1m_2}}\left(\frac{m_2}{m_1}\right)^{it}W\left(\frac{2\pi m_1m_2}{t}\right)+O(T^{-2/3})
\end{equation}
with
$$W(x):=\frac{1}{2\pi i}\int_{(3/2)}x^{-w}e^{w^2}\frac{dw}{w}$$ 
satisfying $W^{(j)}(x)\ll \min(1,x^{-C})$ for any $x>0$ and any fixed $C>0$, $j\in\N$. 
Thus, % since $\mathcal X_{\alpha,\beta}(t)=(t/2)^{-\alpha-\beta}(1+O((|t|+1)^{-1})))$ we have
\begin{equation*}%\label{5genn.1}
\int_{\R}n^{it}\mathcal S_{\alpha,\beta}(t)\phi\Big(\frac{t}{T}\Big)\,dt =\sum_{m_1m_2<T^{1+\delta}}\frac{m_1^{-\alpha}m_2^{-\beta}}{\sqrt{m_1m_2}}\int_{\R} W\left(\frac{2\pi m_1m_2}{t}\right)\left(\frac{nm_2}{m_1}\right)^{it}\phi\Big(\frac{t}{T}\Big)\,dt+O(T^{1/3}).
\end{equation*}
The contribution of the diagonal terms, with $nm_2=m_1$, is
\begin{align}
\mathcal D
&=\frac{1}{ n^{\frac12+\alpha}}\sum_{m<T^{1+\delta}}\frac{1}{m^{1+\alpha+\beta}}\int_{\R}W\left(\frac{2\pi m^2n}{t}\right)\phi\Big(\frac{t}{T}\Big)\,dt\notag\\
&=\frac{1}{ n^{\frac12+\alpha}}\frac{1}{2\pi i}\int_{(3/2)}\frac{e^{w^2}}{w}\left(\frac{1}{2\pi n}\right)^w\zeta(1+2w+\alpha+\beta)\int_{\R}t^w\phi\Big(\frac{t}{T}\Big)\,dt\,dw+O(1)\notag\\
&=\frac{1}{ n^{\frac12+\alpha}}\frac{T}{2\pi i}\int_{(3/2)}\frac{e^{w^2}}{w}\left(\frac{T}{2\pi n}\right)^w\zeta(1+2w+\alpha+\beta)\tilde\phi(w+1)\,dt+O(1).\label{Des}
\end{align}
Next, we consider the off-diagonal terms
\begin{equation*}%\label{5genn.3}
\mathcal O =2\sum_{\substack{m_1m_2<T^{1+\delta} \\ nm_1-m_2\neq 0}}\frac{m_1^{-\alpha}m_2^{-\beta}}{\sqrt{m_1m_2}}\int_{\R}W\left(\frac{2\pi m_1m_2}{t}\right)\left(\frac{nm_2}{m_1}\right)^{it}\phi\Big(\frac{t}{T}\Big)\,dt.
\end{equation*}
We  denote $\Delta:=nm_2-m_1$ and observe that the contribution from $|\Delta|>m_2nT^{-1+2\delta}$ is negligible, since one can integrate by parts with respect to $t$ as above.
 Therefore we can assume $|\Delta| \leq nm_2T^{-1+2\delta}$. We have $m_1=nm_2(1-\frac{\Delta}{nm_2})$ and thus, by Taylor approximation,
\begin{equation}\begin{split}\notag
\frac{1}{m_1^{1/2+\alpha}}&=\frac{1}{(nm_2)^{1/2+\alpha}}\left(1+O\left(\frac{1}{T^{1-2\delta}}\right)\right) \\
\left(\frac{nm_2}{m_1}\right)^{it} &= e^{it\frac{\Delta}{nm_2}}\left(1+O\left(\frac{1}{T^{1-2\delta}}\right)\right)  \\
W\left(\frac{2\pi m_1m_2}{t}\right)&=W\left(\frac{2\pi nm_2^2}{t}\right)+O\left(\frac{1}{T^{1-2\delta}}\right).
\end{split}\end{equation}
Thus, writing $m_1$ in terms of $m_2$ and $\Delta$ we obtain
$$
\mathcal O= \frac{1}{n^{1/2+\alpha}}\sum_{m=1}^{\infty}\frac{1}{m^{1+\alpha+\beta}}\sum_{\Delta\in\Z_{\neq 0}}\int_{-\infty}^{+\infty}W\left(\frac{2\pi nm^2}{t}\right)e^{it\frac{\Delta}{mn}}\phi\Big(\frac{t}{T}\Big)\,dt + O(T^{1/2+6\delta}),
$$
where we could extend back the sums over $\Delta$ and $m$ by integration by parts and by the decay of $W$, respectively. We denote the main term above by $\mathcal O'$. We bypass any issue of convergence with a double integration by parts. We get
\begin{equation*}
\mathcal O'
=\frac{-n^2}{n^{1/2+\alpha}}\sum_{m=1}^{\infty}\sum_{\Delta\in\Z_{\neq 0}}\frac{m^{1-\alpha-\beta}}{\Delta^2}\int_{0}^{+\infty}e^{it\frac{\Delta}{mn}}\frac{d^2}{dt^2}\left(W\left(\frac{2\pi nm^2}{t}\right)\phi\Big(\frac{t}{T}\Big)\right)\,dt .
\end{equation*}
Writing $W$ in terms of its Mellin transform and using $e^{ix}+e^{-ix}=2\cos(x)$, the above becomes
\begin{equation}\begin{split}\notag
\frac{-2n^2}{n^{1/2+\alpha}}\int_{0}^{+\infty}\frac{1}{2\pi i}\int_{(\frac{3}{2})}\frac{e^{w^2}}{w}\sum_{m=1}^{\infty}\sum_{\Delta=1}^{\infty}\frac{m^{1-\alpha-\beta}}{\Delta^2}\cos\left(\frac{t\Delta}{mn}\right)(2\pi nm^2)^{-w}\frac{d^2}{dt^2}\left(t^w\phi\Big(\frac{t}{T}\Big)\right)dw\,dt,
\end{split}\end{equation}
where we could exchange the integrals and the sums, as they converge absolutely. We make the change of variable $\frac{t\Delta}{mn}\to t$ and write
\begin{equation}\begin{split}\notag
\mathcal O'
&=\lim_{\ell\to\infty}\mathcal O'_\ell,
\end{split}\end{equation}
where $\mathcal O'_\ell$ denotes the same expression, but with the integral over $t$ truncated at $\pi(\ell+\frac{1}{2})$ with $\ell\in\N$. In particular,  opening also $\phi$ as a Mellin integral, we get
\begin{equation}\begin{split}\notag
\mathcal O_\ell'
=\frac{-2n}{n^{1/2+\alpha}}\int_{0}^{\pi(\ell+\frac{1}{2})}\frac{1}{2\pi i}\int_{(\frac{3}{2})}\frac{e^{w^2}}{w}&(2\pi)^{-w}\frac{1}{2\pi i}\int_{(-\frac{1}{4})}\tilde\phi(s) T^sn^{-s} \\
& \times\sum_{m=1}^{\infty}\sum_{\Delta=1}^{\infty}\frac{\cos t}{m^{w+s+\alpha+\beta}\Delta^{1+w-s}}\frac{d^2}{dt^2}\left(t^{w-s}\right)ds\,dw\,dt .
\end{split}\end{equation}
By integration by parts, for $w$ and $s$ in the above lines of integration we have 
$$ \int_{0}^{\pi(\ell+\frac{1}{2})}t^{w-s-2}\cos t\,dt = \int_{0}^{\pi(\ell+\frac{1}{2})}\frac{t^{w-s-1}}{w-s-1}\sin t\,dt, $$
so that 
\begin{equation}\begin{split}\notag
\mathcal O_\ell'
&= \frac{-2T}{n^{1/2+\alpha}}\frac{1}{(2\pi i)^2}\int_{(\frac{3}{2})}\int_{(-\frac{1}{4})}\frac{e^{w^2}}{w}(2\pi)^{-w}\tilde\phi(s) \left(\frac{T}{n}\right)^{s-1} \\
&\hspace{2.5cm} \times \zeta(w+s+\alpha+\beta)\zeta(1+w-s)(w-s)\int_{0}^{\pi(\ell+\frac{1}{2})}t^{w-s-1}\sin t\,dt\; ds\,dw. 
\end{split}\end{equation}
We shift the integral over $s$ to $\Re(s)=2$. Notice that in doing so we do not encounter any poles as the singularity of $\zeta(1+w-s)$ is cancelled by the factor $(w-s)$. At this point, the integral over $t$ converges absolutely and so bringing the limit inside we obtain
\begin{equation}\begin{split}\notag
\mathcal O'
&= \frac{-2T}{n^{1/2+\alpha}}\frac{1}{(2\pi i)^2}\int_{(\frac{3}{2})}\int_{(2)}\frac{e^{w^2}}{w}(2\pi)^{-w}\tilde\phi(s) \left(\frac{T}{n}\right)^{s-1}\zeta(w+s+\alpha+\beta)\zeta(1+w-s) \\
&\hspace{4cm} \times(w-s)\int_{0}^{\infty}t^{w-s-1}\sin t\,dt\; ds\,dw.%\\
%&= \frac{4T}{\sqrt n}\frac{1}{(2\pi i)^2}\int_{(\frac{3}{2})}\int_{(2)}\frac{e^{w^2}}{w}(2\pi)^{-w}\tilde\phi(s) \left(\frac{T}{n}\right)^{s-1} \zeta(w+s)\zeta(1+w-s)\\
%&\hspace{4cm} \times \int_{0}^{\infty}t^{w-s}\cos t \,dt\; ds\,dw. 
\end{split}\end{equation}
Being $\Re(1+w-s)=\frac{1}{2}\in(0,1)$,  by~\cite[eq. (3.381.5)]{GR} we have
\begin{align*}
-(w-s)\int_{0}^{\infty}t^{w-s-1}\sin t\,dt&= \int_0^{+\infty} t^{(1+w-s)-1} \cos t \,dt \\
&= \Gamma(1+w-s)\cos\left( \frac{\pi}{2}(1+w-s) \right).
\end{align*}
Thus,
\begin{equation}\begin{split}\notag
\mathcal O'
&= \frac{2T}{n^{1/2+\alpha}}\frac{1}{(2\pi i)^2}\int_{(\frac{3}{2})}\int_{(2)}\frac{e^{w^2}}{w}(2\pi)^{-w}\tilde\phi(s) \left(\frac{T}{n}\right)^{s-1} \\
&\hspace{2.5cm} \zeta(w+s+\alpha+\beta)\zeta(1+w-s)\Gamma(1+w-s)\cos\left(\frac{\pi}{2}(1+w-s)\right) ds\,dw\\
&= \frac{T}{n^{1/2+\alpha}}\frac{1}{(2\pi i)^2}\int_{(\frac{3}{2})}\int_{(2)}\frac{e^{w^2}}{w}\tilde\phi(s)\left(\frac{T}{2\pi n}\right)^{s-1} 
\zeta(w+s+\alpha+\beta)\zeta(s-w)ds\,dw,
\end{split}\end{equation}
by the functional equation.
Finally, we shift the integral over $s$ to the line $\Re(s)=3$. The contribution of minus the residue at $s=1+w$ is 
\begin{equation}\begin{split}\notag 
-\frac{T}{n^{1/2+\alpha}}\frac{1}{2\pi i}\int_{(\frac{3}{2})}\frac{e^{w^2}}{w}\tilde\phi(1+w)\left(\frac{T}{2\pi n}\right)^{w} 
\zeta(1+2w+\alpha+\beta)dw  =-\mathcal D+O(1)
\end{split}\end{equation}
with $\mathcal D$ as in~\eqref{Des}. Therefore, we have
\begin{equation*}
\mathcal O
= \frac{T}{n^{1/2+\alpha}}\frac{1}{(2\pi i)^2}\int_{(\frac{3}{2})}\int_{(3)}\frac{e^{w^2}}{w}\tilde\phi(s)\left(\frac{T}{2\pi n}\right)^{s-1} 
\zeta(w+s+\alpha+\beta)\zeta(s-w)ds\,dw - \mathcal D + O(T^{1/2+6\delta}).
\end{equation*}
Thus, after a change of variable in $s$, we obtain
\begin{equation}
\begin{split}
\int_{\R}n^{it}\mathcal S_{\alpha,\beta}(t)\phi\Big(\frac{t}{T}\Big)\,dt &=T\frac{(T/2\pi)^{-\frac{\alpha+\beta}2}}{n^{\frac{1+\alpha-\beta}{2}}}\frac{1}{(2\pi i)^2}\int_{(\frac{3}{2})}\int_{(3)}\frac{e^{w^2}}{w}\tilde\phi(s-\tfrac{\alpha+\beta}2)\left(\frac{T}{2\pi n}\right)^{s-1} \\
&\hspace{5em}\times\zeta(s+\tfrac{\alpha+\beta}2+w)\zeta(s-\tfrac{\alpha+\beta}2-w)ds\,dw 
+ O(T^{1/2+6\delta}).
\end{split}\label{fact}
\end{equation}
Now, we have $\mathcal X_{\alpha,\beta}(t)=(t/2\pi)^{-\alpha-\beta}(1+O(1/T))$ and thus by the above computation we find
\begin{align*}
&\int_{\R}n^{it}\mathcal S_{-\beta,-\alpha}(t)\mathcal X_{\alpha,\beta}(t)\phi\Big(\frac{t}{T}\Big)\,dt =(T/2)^{-\alpha-\beta}\int_{\R}n^{it}\mathcal S_{-\beta,-\alpha}(t)\mathcal (t/T)^{-\alpha-\beta}\phi\Big(\frac{t}{T}\Big)\,dt+O(1)\\
&\hspace{5em}=T\frac{(T/2\pi)^{-\frac{\alpha+\beta}2}}{n^{\frac{1+\alpha-\beta}{2}}}\frac{1}{(2\pi i)^2}\int_{(\frac{3}{2})}\int_{(3)}\frac{e^{w^2}}{w}\tilde\phi(s-\tfrac{\alpha+\beta}2)\left(\frac{T}{2\pi n}\right)^{s-1} \\
&\hspace{12em}\times\zeta(s-\tfrac{\alpha+\beta}2+w)\zeta(s+\tfrac{\alpha+\beta}2-w)ds\,dw + O(T^{1/2+6\delta}).
\end{align*}
Making the change of variable $w\mapsto -w$ and summing with~\eqref{fact} we obtain
\begin{equation*}%\label{fac}
\begin{split}
I_n &=\frac{T}{n^{\frac{1}2+\alpha}}\frac{1}{2\pi i}\int_{(3)}\tilde\phi(s-\tfrac{\alpha+\beta}2)\left(\frac{T}{2\pi n}\right)^{s-\frac{\alpha+\beta}2-1} \zeta(s+\tfrac{\alpha+\beta}2)\zeta(s-\tfrac{\alpha+\beta}2)ds 
+ O(T^{1/2+6\delta}).
\end{split}
\end{equation*}
 by~\eqref{afe1} and the residue theorem. Proposition~\ref{TwSM} then follows by taking $\delta<\eps/6$.

\section{The twisted fourth moment}\label{SL}
Proposition~\ref{p4m}  follows as in~\cite{BBLR} from the following refinement of Theorem~1.3 of~\cite{BBLR}.

\begin{prop}\label{qdpt}
Let $A,B,X,Z,T\geq 1$ with $Z>XT^{-\eps}$ and $\log(ABXZ)\ll \log T$. Let $\boldsymbol\alpha_a,\boldsymbol\beta_b$ be sequences of complex numbers supported on $[1,A]$ and $[1,B]$, respectively, and such that $\boldsymbol\alpha_a\ll A^\eps, \boldsymbol\beta_b\ll B^\eps$. Let $f\in\mathcal C^{\infty}(\R_{\geq0}^{3})$ and $K\in\mathcal C^{\infty}(\R_{\geq0})$  be such that
\begin{equation*}
\frac{\partial ^{i+j+k}}{\partial x^i\partial y^j\partial h^k}f(x,y,h)\ll_{i,j,k,r} T^{\eps}(1+x)^{-i}(1+y)^{-j}(1+h)^{-k}(1+h^2Z^2/(xy))^{-r}
\end{equation*}
and $K^{(j)}(x)\ll_{j,r} T^\eps (1+x)^{-j}(1+x/X^2)^{-r}$ for any $i,j,k,r\geq0$. Then,  
\begin{align*}
& \sum_{a,b,m_1,m_2,n_1,n_2,h>0\atop a m_1 m_2 - b n_1 n_2 = h } \frac{\boldsymbol\alpha_a \overline{\boldsymbol\beta_b}}{m_1^{\alpha} m_2^{\beta} n_1^{\gamma} n_2^{\delta}} 
 f(a m_1 m_2, b n_1 n_2, h) K(m_1 m_2 n_1 n_2)\\
&\hspace{15em}= \mathcal{M}_{\alpha, \beta, \gamma, \delta} + \mathcal{M}_{\beta, \alpha, \gamma, \delta} + \mathcal{M}_{\alpha, \beta, \delta, \gamma} + \mathcal{M}_{\beta, \alpha, \delta, \gamma} + \mathcal{E}
\end{align*}
where $\mathcal M$ is as in~\cite[p.~21]{BBLR} and $\mathcal E \ll T^\eps(AB)^\frac12 XZ^{-\frac12}\Big(AB + (ABX)^{\frac58}Z^{-\frac12}+(AB)^{\frac12}XZ^{-\frac 34}\Big).$
\end{prop}

As only few changes are needed to the proof of Theorem~1.3 of~\cite{BBLR}, we will not repeat the full argument here, but only indicate the changes. The main difference is that we slightly refine~\cite[Lemma~3.2]{BBLR} (which is proven in~\cite[Proposition~3]{BCR} and is essentially due to Watt~\cite[Proposition~4.1]{Watt}).

\begin{lemma}\label{Watt}
Let $H,C,R,S, V,P\geq1$, $\delta\leq1$ and let $X:=\sqrt{\frac {RSVP}{HC}}$.
Moreover, assume that $\alpha(y), \beta(y)$ and  $\gamma_{r,s}(x,y)$ (for any $r,s\in\Z$) are complex valued smooth functions, supported on the intervals $[1,H]$, $[1,C]$ and $[V,2V]\times[P,2P]$, respectively, such that
$\alpha^{(j)}(x),\beta^{(j)}(x)\ll_j (\delta x)^{-j}$ and $\frac{\partial^{i+j}}{\partial x^i\partial y^j}\gamma_{r,s}(x,y)\ll_{i,j} x^{-i}y^{-j}$ for any $i,j\geq0$. Assume $a_r,b_s$ are sequences of complex numbers supported on $[R,2R]$, $[S,2S]$, respectively, and such that $a_r\ll R^\varepsilon$, $b_s\ll S^\varepsilon$ for any $\eps>0$. Then 
\begin{align}\label{wbou}
\Sigma&:=\sum_{\substack{h,c,r,s,v,p\\(rv,sp)=1}}\alpha(h)\beta(c)\gamma_{r,s}(v,p)a_{r}b_{s} \textrm{e}\bigg(\pm\frac{hc\overline {rv}}{sp}\bigg)\notag\\
& \ll \delta^{-10} HC  R(V+SX)\Big(1+\frac{HC}{RS}\Big)^\frac12\Big(1+\frac{P}{RV} \Big)^\frac12\Big(1+\frac{H^2CPX^2}{R^4S^3V}\Big)^\frac14 (HCRSVP)^{\varepsilon}.
\end{align}
\end{lemma}
\begin{proof}
Lemma~3.2 of~\cite{BBLR} gives the above bound (with $\delta^{-7/2}$ in place of $\delta^{-10}$) under the hypotheses
\begin{equation}\label{cowa}
X\gg (RSVP)^{\varepsilon},\qquad (RS)^2\geq\max\Big\{H^2C,\frac{SP}{V}(RSVP)^{\varepsilon}\Big\}.
\end{equation}
 Let's now show that these hypotheses can be dropped. We let $Y:=HCRSVP$. Applying Poisson summation formula in $h$ and $c$ one easily sees that 
$$\Sigma\ll Y^{\eps}\delta^{-2}\cdot (HCRV+(PS)^2RV)=Y^{\eps}\delta^{-2}(HCRV+X^2 HCPS).
$$
If $X\ll Y^\eps$ this is stronger than~\eqref{wbou} since $HC  R\cdot V=HCRV$ and 
$HCR\cdot SX\cdot (\frac{H C}{RS})^{1/2}(\frac{P}{RV})^{1/2}=HCP S\cdot  R^{1/2}$. Thus, the assumption $X\gg Y^\eps$ can be dropped.
Moreover, with the above notation~\cite[Theorem~12]{DI} can be rewritten as 
\begin{align*}
\Sigma&\ll \delta^{-10}Y^{\eps}\sqrt{C H R S} \Big(\sqrt{   P^2 S \sqrt{R (C H + R S)} V} + \sqrt{   C H R V^2/S} + \sqrt{   P S (C H + R S) (P + R V)}\Big)\\
%&=C H R  \Big(\sqrt{   P^2 S^2 \sqrt{R (C H + R S)} V/CHR} +V+ \sqrt{   P S^2 (C H + R S) (P + R V)/CHR}\Big)\\
%&=C H R  \Big(SX\frac{P^{1/2}}{R^{1/2}S^{1/4}} (1+C H / R S)^{1/4} +V+SX (  1+C H / R S)^{1/2} (1+P / R V)^{1/2}\Big)\\
%&=C H R  \Big(V+SX (  1+C H / R S)^{1/2} (1+P / R V)^{1/2}(1+\frac{\frac{P^{1/2}}{R^{1/2}S^{1/4}}}{(1+C H / R S)^{1/4}  (1+P / R V)^{1/2}})\Big)\\
&=\delta^{-10}Y^{\eps} H CR  \bigg(V+SX \Big(1+\frac{HC}{RS}\Big)^\frac12\Big(1+\frac{P}{RV} \Big)^\frac12\Big(1+\Big(B\frac{H^2CPX^2}{R^{4}S^{3}V}\Big)^{1/4}\Big)\bigg),
   \end{align*}
with 
$$
B:=\frac{RS/H}{(1+C H / R S)  (1+P / R V)^{2}}.
$$
Thus, this bound is clearly at least as strong as~\eqref{wbou} if $B\leq 1$. If $(RS)^2\leq H^2C$ then $B\leq \frac{(RS)^{2}}{CH^2}\leq1$, whereas if $(RS)^2\leq \frac{SP}{V}$ then $B\leq \frac{RS(RV)}{HP}\leq  \frac{1}{H}\leq1$. Thus, the second assumption in~\eqref{cowa} can also be dropped.
\end{proof}

This lemma allows for a simplification of the computations in~\cite[pp.~16--20]{BBLR} since we can now apply the above bound to the full sums. As in~\cite{BBLR} we let $W_0,\dots W_4$ be smooth function supported in $[1,2]$ and such that $W_i^{(j)}\ll_j (ABMN)^{\eps}$ for any fixed $j\geq0$. Also, let  $\boldsymbol \alpha_a,\boldsymbol\beta_b$ be sequences in $\C$ supported on $[A,2A]$ and $[B,2B]$, and such that $\boldsymbol\alpha_a\ll A^\eps, \boldsymbol\beta_b\ll B^\eps$.  Also, let $M_1,M_2,N_1,N_2,H\geq1$ and let $M=M_1M_2$, $N=N_1N_2$. Then, under the assumptions
\begin{align*}
M_1\leq M_2(AM)^\varepsilon,\quad N_1\leq N_2(AM)^\varepsilon,\quad BN_1\leq AM_1,\quad AM\asymp BN
\end{align*}
applying Lemma~\ref{Watt} and following the same simple computations of~\cite[p.~20]{BBLR} we obtain for $x\asymp \frac {dN_2}{AM_1}$,
\begin{align*}
Z_{\pm,d}(x)&:=\sum_{\substack{a,b,m_1,n_1,h\\(am_1,bn_1)=d}}\sum_{0<|l|\leq \frac{AM}{dM_2N_2}(AM)^\eps}\boldsymbol \alpha_a\boldsymbol\beta_bW_0\Big(\frac {dh}H\Big)W_1\Big(\frac{m_1}{M_1}\Big)W_3\Big(\frac{n_1}{N_1}\Big)\\
&\qquad\qquad W_2\Big(\frac{bn_1x}{dM_2}\Big)W_4\Big(\frac{am_1x}{dN_2}\Big)  \textrm{e}\bigg(\mp lh\frac{\overline{am_1/d}}{bn_1/d}\bigg)e(lx)\\
&\ll (HAM)^{\eps}\frac{A^2BH^\frac12}{d^2}  \Big(\frac{M_1N_1}{M_2N_2}\Big)^{\frac12 } (BM)^\frac12 \Big(1+\frac{N_1^2H}{A^3B^2}\Big)^\frac14\Big(1+\frac{H}{(AB)^\frac12}\Big)^\frac12,
\end{align*}
where with respect to~\cite[eq.~(18)]{BBLR} we kept the extra factor of $(1+\frac{H}{(AB)^{1/2}})^{1/2}$ (coming from the factor $(1+\frac{HC}{RS})^{1/2}$ in~\eqref{wbou}) but we now have no hypothesis on $d$ and $H$.

Thus, as in~\cite[p.~17]{BBLR} (but with no need to treat separately the sums of large and small $d$s) we deduce
\begin{align*}
\sum_{d\leq 2H} \int_{x\asymp \frac{dN_2}{AM_1}}|Z_{\pm,d}(x)|\,dx  &\ll (AMH)^{\frac12 +\eps}\Big(AB + (ABH)^\frac14(ABMN)^\frac18\Big)\bigg(1+\frac{H}{(AB)^\frac12}\bigg)^{\frac12}\\
&\ll  (AMH)^{\eps}E
\end{align*}
with
\begin{align*}
&E:= (ABMNH^2)^{\frac14}\Big(AB + (ABH)^\frac14(ABMN)^\frac18+(AB)^{\frac34}H^{\frac12}+H^{\frac 34}(ABMN)^\frac18\Big).
\end{align*}
The rest of the proof then follows unchanged. In particular, Proposition~3.1 of~\cite{BBLR} now holds with $\mathcal E\ll(E+H^2)(ABMNH)^\eps$ (with no conditions on $H$) and one then deduces  Theorem~4.1 and Corollary~4.1  (again with no need to split the sums) with the stronger bounds
$$
\mathcal E\ll T^\eps(ABX^2H^2)^{\frac14}\Big(AB + (ABH)^\frac14(ABX^2)^\frac18+(AB)^{\frac34}H^{\frac12}+H^{\frac 34}(ABX^2)^\frac18\Big)+T^\eps H^2
$$
and 
\begin{align*}
\mathcal E%&\ll T^\eps(ABX^2/Z)^{\frac12}\Big(AB + (ABX/Z)^\frac12(ABX^2)^\frac18+(AB)^{\frac34}(ABX^2/Z^2)^{\frac14}+(ABX^2/Z^2)^{\frac 38}(ABX^2)^\frac18\Big)+ABX^2Z^2T^\eps.\\
&\ll T^\eps(AB)^\frac12 XZ^{-\frac12}\Big(AB + (ABX)^{\frac58}Z^{-\frac12}+(AB)^{\frac12}XZ^{-\frac 34}\Big).%\\
%&\ll T^\eps(AB)^\frac12 XZ^{-\frac12}\Big(AB +(AB)^{\frac12}X^{\frac56}Z^{-\frac 7{12}}\Big).%Slightly weaker
\end{align*}
respectively. Proposition~\ref{qdpt} then follows.

\appendix
\section{}\label{Appendix}
In this appendix we prove that the kernel $F_k$ defined in \cite[Theorem 5.2]{HughesThesis} equals $W_U^k$ in~\eqref{w1ld}. More precisely, we show that $\hat F_k(y) = \hat W_U^{k}(y)$ for $y\in(0,1)$, the other ranges being analogous. By direct computation, one can easily see that 
\begin{equation}\begin{split}\label{kernelHughes}
\hat F_k(y) =    
&\sum_{i=0}^{k-1}\sum_{n=0}^{k-i-1}
\frac{(-1)^{k}b_{2n}}{(2k-2n-1)!} \binom{2k-2n-1}{2i}y^{2i}  \\
&+\sum_{i=0}^{k-1}\sum_{n=1}^{k-i}
\frac{(-1)^{k-1}b_{2n-1}}{(2k-2n)!} \binom{2k-2n}{2i}y^{2i}
+\sum_{i=1}^{k} \frac{(-1)^{k}c_{2k-2i}}{2(2i-1)!}y^{2i-1}
\end{split}\end{equation}
for $y\in(0,1)$, with $b_r$ and $c_{2r}$ as defined in \cite{HughesThesis}, Equations (5.20) and (5.21) respectively.
In the same range for $y$, by using the properties in \cite[Conjecture 2.1]{4.}, we get
\begin{equation}\begin{split}\label{kerneltesi}
\hat W_U^k(y)
= -k - \frac{k}{2}\left( (k+1)\sum_{j=1}^k (-1)^j d_{j,k}\frac{y^{2j-1}}{2j-1} + (k-1)\sum_{j=1}^{k-1}(-1)^jd_{j,k-1}\frac{y^{2j-1}}{2j-1}\right)
\end{split}\end{equation}
with 
$ d_{j,k} = \frac{1}{j}\binom{k-1}{j-1}\binom{k+j}{j-1}\footnote{Note the slight change of notation for the coefficients $d_{j,k}$, which are denoted $c_{j,k}$ in \cite{FazzariThesis}.}.$
Now we prove that the coefficients of the polynomial~\eqref{kerneltesi} equal those of~\eqref{kernelHughes}. 
Let's start with the coefficient of $y^{2k-1}$; the claim is
\begin{equation}\label{coeff2k-1}  
\frac{2b_{0,k}}{(2k-2)!} = (k+1)  \binom{2k}{k-1}.\end{equation}
Note that, since 
$\binom{k}{1-n} = 0$ for $n\geq 2$, we have
\begin{equation}\begin{split}\notag
b_{0,k} =  -\frac{(2k-1)!}{(k-1)!}\frac{(2k-1)!}{(k-1)!} + \frac{(2k-2)!}{(k-1)!}\frac{(2k)!}{(k-1)!},
\end{split}\end{equation}
then~\eqref{coeff2k-1} follows by direct computation. 
Next, we show that the constant terms are equal, that amounts to proving that
\begin{equation}\label{constantterm2} 
(-1)^k\sum_{i=0}^{2k-1}\frac{(-1)^ib_{i}}{(2k-i-1)!}  = -k .
\end{equation}
By definition of $b_i$, the left-hand side above can be written as
\begin{equation}\notag
(-1)^k\sum_{n=0}^k \frac{(2k-n-1)!}{k!(k-1)!}\binom{k}{n} \sum_{i=0}^{2k-1}\frac{(-1)^i(2k+n-i-1)!}{(2k-i-1)!}\binom{k}{i+1-n}(2n-i-1),
\end{equation}
therefore~\eqref{constantterm2} easily follows from
\begin{equation}\begin{split}\notag\label{22june.1}
 \sum_{i=0}^{2k-1}\frac{(-1)^i(2k+n-i-1)!}{(2k-i-1)!}\binom{k}{i+1-n}(2n-i-1) 
 = \begin{cases}
 0 &\text{if }n<k-1 \\
(-1)^kk! &\text{if }n=k-1 \\
(-1)^{k+1}k(k+1)! &\text{if }n=k.
\end{cases}
\end{split}\end{equation}
To prove the above, we notice that it is equivalent to showing
\begin{equation}\label{0001}  S:=\sum_{j=0}^{k}\frac{(-1)^j(2k-j)!}{(2k-j-n)!}\frac{n-j}{j!(k-j)!} 
 = \begin{cases}
 0 &\text{if }n<k-1 \\
1 &\text{if }n=k-1 \\
k(k+1) &\text{if }n=k,
\end{cases}\end{equation}
where we used that $\binom{k}{i+1-n}=0$ unless $0\leq i+1-n\leq k$ and made the change of variable $j:=i+1-n$.
The left hand side above can be written as
$$ S=nS_1-S_2 $$
with $$ S_1= \sum_{j=0}^{k}\frac{(-1)^j(2k-j)!}{(2k-j-n)!}\frac{1}{j!(k-j)!},
\quad  S_2 = \sum_{j=0}^{k}\frac{(-1)^j(2k-j)!}{(2k-j-n)!}\frac{j}{j!(k-j)!}. $$ 
Then~\eqref{0001} follows from the Gauss summation theorem, which yields
\begin{equation}\notag
S_1 
= \frac{(2k)!}{k!(2k-n)!}\sum_{j=0}^k\frac{(-k)_j(n-2k)_j}{(2k)_jj!}
= \frac{(2k)!}{k!(2k-n)!} \pFq{2}{1}{-k,n-2k}{-2k}{1}
= \begin{cases} 0 &\text{if }n<k \\ 1 &\text{if }n=k \end{cases}
\end{equation}
and 
\begin{equation}\notag
S_2 
=-\frac{(2k)!}{2k!(2k-n-1)!}\pFq{2}{1}{1-k,n-2k+1}{1-2k}{1}
= \begin{cases} 0 &\text{if }n<k-1 \\ 
-1  &\text{if }n=k-1  \\  
-k^2 &\text{if }n=k. \end{cases}
\end{equation}

Finally we need to show that the coefficients of $y^{\ell}$ with $0<\ell<2k-1$ coincide. We start with the case of $\ell=2i$ even. In this case the problem amounts to proving that the coefficient of $y^{2i}$ in $\hat F_k(y)$ is zero. Equivalently, by the definition of $b_r$ and proceeding as before, we need to show that
\begin{equation}\label{27june.1}
\sum_{n=0}^k\frac{(-1)^{n-1}(2k-n-1)!}{k!(k-1)!}\binom{k}{n} 
\sum_{m=0}^{k} \frac{(-1)^{m}(2k-m)!}{(2k-m-n)!} \binom{2k-m-n}{2i}
\binom{k}{m}(n-m)=0.
\end{equation}
Let's focus on the inner sum, which we decompose as
\begin{equation}\begin{split}\notag 
\sum_{m=0}^{k} \frac{(-1)^{m}(2k-m)!}{(2k-m-n)!} \binom{2k-m-n}{2i}
\binom{k}{m}(n-m) = nS_1 - S_2,
\end{split}\end{equation}
where, again by the Gauss summation theorem,
$$ S_1 
=  \sum_{m=0}^{k} \frac{(-1)^{m}(2k-m)!}{(2k-m-n)!} \binom{2k-m-n}{2i}
\binom{k}{m}
=\frac{(2k)!}{(2k-n)!}\binom{2k-n}{2i}\frac{(-n-2i)_k}{(-2k)_k}  $$
and
$$ S_2 
= \sum_{m=0}^{k} \frac{(-1)^{m}(2k-m)!}{(2k-m-n)!} \binom{2k-m-n}{2i}
\binom{k}{m}m
= -\frac{(2k)!(2k-n-2i)}{2(2k-n)!}\binom{2k-n}{2i}\frac{(-n-2i)_{k-1}}{(1-2k)_{k-1}}. $$
Therefore the left hand-side of~\eqref{27june.1} can be written as 
\begin{equation}\begin{split}\label{20july.1}
&\frac{k}{(-2k)_k}\binom{2k}{k}\sum_{n=0}^k \frac{(-1)^{n-1}}{2k-n}\binom{k}{n}\binom{2k-n}{2i}n(-n-2i)_k\\
&\hspace{1cm}-\frac{k}{(-2k)_k}\binom{2k}{k}\sum_{n=0}^k \frac{(-1)^{n-1}}{2k-n}\binom{k}{n}\binom{2k-n}{2i}(2k-n-2i)k(-n-2i)_{k-1}.
\end{split}\end{equation}
Hence,~\eqref{27june.1} follows from the two identities
\begin{equation}\begin{split}\notag
&\text{(i)}\quad\;\sum_{n=0}^k \frac{(-1)^nn}{2k-n}\binom{k}{n}\binom{2k-n}{2i}(-n-2i)_k=0\\
&\text{(ii)}\quad\sum_{n=0}^k \frac{(-1)^n(2k-n-2i)}{2k-n}\binom{k}{n}\binom{2k-n}{2i}(-n-2i)_{k-1}=0
\end{split}\end{equation}
for all $0<i<k$, which can be proven by direct computation. For example in the case $k\leq 2i\leq 2k$ (the other is analogous), (i) and (ii) are equivalent respectively to 
$$\pFq{3}{2}{2i+2,2i-2k+1,1-k}{2-2k,2i-k+2}{1}=0 \quad \text{and}\quad
\pFq{3}{2}{2i+1,2i-2k+1,-k}{1-2k,2i-k+2}{1}=0,$$
 for $k\leq 2i\leq 2k$, and these can be shown by Dixon and Watson's theorems (see Equation (1) p. 16 and Equation (1) p. 13 from \cite{Bailey}).

Now, we consider the case $\ell=2i-1$ odd, with $0<i<k$. We want to show that
\begin{equation}\label{oddcoeff} 
\frac{(-1)^{k}c_{2k-2i}}{2(2i-1)!}
=-\frac{(-1)^{k}}{(2i-1)!} \sum_{m=0}^{2k-2i}\frac{(-1)^mb_m}{(2k-2i-m)!}.
\end{equation}
After our now-familiar manipulations, we can rewrite the left-hand side above as $I_1+I_2$, with 
\begin{equation}\begin{split}\notag
I_1 &=\frac{k}{k!}\frac{(-1)^{k}}{(2i-1)!} \sum_{n=0}^k \frac{n(-1)^n(2k-n-1)!}{n!(k-n)!}(n+2i-1)!\binom{k}{n+2i-k-1}\\
I_2&=\frac{k}{k!}\frac{(-1)^{k}}{(2i-1)!} \sum_{n=0}^k \frac{k(-1)^n(2k-n-1)!}{n!(k-n)!} (n+2i-1)!\binom{k}{n+2i-k}.
\end{split}\end{equation}
Equation~\eqref{oddcoeff} then reads
\begin{equation}\begin{split}\notag
\frac{i(-1)^{k+i+1}}{k!}\frac{(k-i)!}{(k+i-1)!}&\frac{1}{\binom{2i-2}{i-1}} \sum_{n=0}^k\frac{(-1)^n(2k-n-1)!(n+2i-1)!}{n!(k-n)!}\\
&\times\bigg(n\binom{k}{n+2i-k-1}+k\binom{k}{n+2i-k}\bigg)=1,
\end{split}\end{equation}
or, for example for $k\leq2i\leq2k$, equivalently
\begin{equation}\begin{split}\notag
\frac{i(-1)^{k+i+1}}{k!}\frac{(k-i)!}{(k+i-1)!}&\frac{1}{\binom{2i-2}{i-1}} 
\bigg\{-(2i)!\frac{(2k-2)!}{(k-1)!}\binom{k}{2i-k}\pFq{3}{2}{1+2i,2i-2k,1-k}{2-2k,1+2i-k}{1}\\
&\hspace{1.5cm}+(2i-1)!\frac{(2k-1)!}{(k-1)!}\binom{k}{2i-k}\pFq{3}{2}{2i,2i-2k,-k}{1-2k,1+2i-k}{1}\bigg\}=1.
\end{split}\end{equation}
Evaluating the remaining hypergeometric functions by Dixon's and Watson's theorem again, the above identity is proven.

\end{document}